\documentclass[11pt,reqno]{amsproc}

\title[Global Existence of Weak Solutions \\ for Compresssible
Navier--Stokes--Fourier Equations  \\ with the Truncated 
Virial Pressure Law]{Global Existence of Weak Solutions for \\ Compresssible
Navier--Stokes--Fourier Equations \\  with the Truncated Virial Pressure Law}

\author[D.~Bresch]{Didier Bresch}
\address{UMR 5127 CNRS, Univ. Savoie Mont Blanc, LAMA, Universit\'e Savoie Mont Blanc, 73376 Le Bourget du Lac, France}
\email{didier.bresch@univ-smb.fr}

\author[P.~Jabin]{Pierre--Emmanuel Jabin}
\address{Department of Mathematics and Huck Institutes, Pennsylvania State University, State College, PA 16801, USA}
\email{pejabin@psu.edu}

\author[F.~Wang]{Fei Wang}
\address{School of Mathematical Sciences, CMA-Shanghai, Shanghai Jiao Tong University, 
Shanghai 200240,China}
\email{fwang256@sjtu.edu.cn}

\usepackage{fancyhdr}
\usepackage[margin=1.25in]{geometry}
\usepackage{amsmath, amsthm, amssymb}
\usepackage{graphicx}

\usepackage[usenames,dvipsnames,svgnames,table]{xcolor}
\usepackage[colorlinks=true, pdfstartview=FitV, linkcolor=blue, citecolor=blue, urlcolor=blue]{hyperref}

\begin{document}

\def\epsilon{\varepsilon}
\def\intint{\int\!\!\!\!\int}
\def\OO{\mathcal O}
\def\SS{\mathcal S}
\def\RR{\mathbb R}
\def\TT{\mathbb T}
\def\ZZ{\mathbb Z}
\def\HH{\mathbb H}
\def\RSZ{\mathcal R}
\def\GG{\mathcal G}
\def\eps{\varepsilon}
\def\tt{\langle t\rangle}
\def\erf{\mathrm{Erf}}
\def\red#1{\textcolor{red}{#1}}
\def\blue#1{\textcolor{blue}{#1}}
\def\mgt#1{\textcolor{magenta}{#1}}
\def\ff{\rho}
\def\gg{\gamma_{art}}
\def\ggg#1{\gamma_{art,#1}}
\def\phiyy{\partial_{yy} \phi}
\def\tilde{\widetilde}
\def\sqrtnu{\sqrt{\nu}}
\def\ww{w}
\def\ft#1{#1_\xi}
\def\cl{\mathcal{L}}
\def\ck{\mathcal{K}}
\def\ce{\mathcal{E}}
\def\cs{\mathcal{S}}
\def\fei#1{\textcolor{Magenta}{#1}}
\def\nn{\nonumber\\}
\def\sp{\ \ \ \ \ \ \ \ \ \ \ \ \ \ \ \ \ \ \ \ \ \ \ \ \ \ \ \ \ \ \ \ \ \ \ \ \ \ \ \ \ \ \ }
\def\ss{s_{0}}
\def\TH{\vartheta_{2}}
\def\TTT{\vartheta_{3}}
\def\bep{C\lambda^{-1}}
\def\yy{\tilde y}
\def\dq{\qquad}
\def\DD{\mathcal{D}}

\def\ee{\tilde e}
\def\pp{\tilde P}
\def\nn{\nonumber\\}

\newtheorem{theorem}{Theorem}[section]
\newtheorem{corollary}[theorem]{Corollary}
\newtheorem{proposition}[theorem]{Proposition}
\newtheorem{lemma}[theorem]{Lemma}

\theoremstyle{definition}
\newtheorem{definition}{Definition}[section]
\newtheorem{remark}[theorem]{Remark}

\def\theequation{\thesection.\arabic{equation}}
\numberwithin{equation}{section}

\def\dist{\mathop{\rm dist}\nolimits}    
\def\sgn{\mathop{\rm sgn\,}\nolimits}    
\def\Tr{\mathop{\rm Tr}\nolimits}    
\def\div{\mathop{\rm div}\nolimits}    
\def\supp{\mathop{\rm supp}\nolimits}    
 \def\indeq{\qquad{}\!\!\!\!}                     
\def\period{.}                           
\def\semicolon{\,;}                      

\def\nts#1{{\cor #1\cob}}
\def\cor{\color{red}}
\def\cog{\color{green}}
\def\cob{\color{black}}
\def\coe{\color{blue}}
\def\ao{a_0}

\def\namedlabel#1#2{\begingroup
	#2%
	\def\@currentlabel{#2}%
	\phantomsection\label{#1}\endgroup
}

\def\fei#1{\textcolor{violet}{#1}}

\begin{abstract}
  This paper concerns the  existence of global  weak solutions {\it \`a la Leray}  for compressible Navier--Stokes--Fourier system with periodic boundary conditions  and the truncated virial pressure law which is assumed to be  thermodynamically unstable. More precisely, the main novelty is that  the pressure law is not assumed to be monotone with respect to the density. This provides the first global weak solutions result for the compressible Navier-Stokes-Fourier system with such kind of pressure law which is strongly used as a generalization of the perfect gas law. The paper is based on a new construction of approximate solutions through an iterative scheme and fixed point procedure which could be very helpful to 
design efficient numerical schemes. Note that our method involves the recent paper by the authors published in Nonlinearity (2021) for the compactness of the density when the temperature is given.
\end{abstract}

\maketitle


\medskip

\noindent{\bf AMS Classification numbers:}  76N10, 35D30, 35Q30, 35Q86

\section{Introduction and main result}
\smallskip
This paper is dedicated to Anton\'{\i}n Novotn\'y who had contributed so many innovative work to the theory of compressible fluids, specifically to compressible Navier-Stokes-Fourier equations, and unfortunately  passed away suddenly on Thursday, June 03 2021. 

The non-stationary Navier-Stokes-Fourier equations modeling viscous compressible and heat conducting fluids, in the multi-dimensional in space case, have been extensively studied both from a theoretical and a numerical point of view: see \cite{FeKaPo}. Yet many questions around the existence, uniqueness, or stability of solutions have remained unsolved. The case of non-stationary barotropic Navier-Stokes equations (namely without temperature) is somewhat better understood, in particular for the global existence of weak solutions \`a la Leray (\cite{Le}): see for instance  \cite{Li}, \cite{Fe},  \cite{FeNoPe}, \cite{Fe0}, \cite{PlWe},  \cite{BreJab18}, \cite{BrJaWa1} and references cited therein. The present study addresses the theoretical problem of existence of so-called global weak solutions \`a la Leray for the full system including the evolution of internal energy (temperature dependent case) for the so-called virial pressure law.

One of the well-known difficulty of such nonlinear system of fluid mechanics with heat-conductivity is that the {\it a priori} bounds based on the energy estimates are not strong enough to get equi-integrability of certain quantities, such as the viscous dissipation quantity (see for instance \cite{Li}). This is compounded in the present paper by a pressure law that is non-monotone in the density and hence thermodynamically unstable.

A first helpful approach is to  replace the internal energy equation by the entropy inequality supplemented by the total energy balance, as introduced by E. Feireisl and A. Novotny, explained in \cite{Febook}, \cite{FeNo} with appropriate hypothesis on the pressure state laws. Unfortunately, this approach was initially limited to thermodynamically stable state laws, namely 
\[
\partial_\rho P\vert_{\vartheta}  > 0, \qquad \qquad \partial_\vartheta e\vert_{\rho} > 0,
\]
where $P$ is the pressure state law and $e$ is the internal energy depending on the density $\rho$ and the temperature $\vartheta$.

On the other hand, the potential oscillations in the density due to the pressure laws can in principle be controlled through the method in \cite{BreJab18, BrJaWa1}. But a major difficulty further lies in combining both approaches at the level of  an approximate system. We take a different point of view to bypass most this issue by constructing solutions through a fixed point argument.

\smallskip

Define, in a periodic domain $\Omega = \TT^d$ for $d\ge2$, the so-called truncated virial pressure law 
\begin{equation}\label{virial}
P(\rho,\vartheta)= \rho^\gamma + \vartheta\,\sum_{n=0}^N B_n(\vartheta)\, \rho^n 
\end{equation} 
where $\gamma>\max(4, 2N, d)$. The virial equation of state seems to have been proposed first by M.~Thiesen in 1885 and intensively studied by H. Kammerlingh Onnes (see \cite{Ow}) at the beginning of the previous century as an empirical extension of the ideal-gas law. The reader interested by Virial coefficients of pure gases and mixtures is referred to \cite{DyWi}.

Such  pressure laws is not monotone with respect to the density even after a fixed value and therefore is not thermodynamically stable. They are nevertheless commonly used in practice. With proper assumptions on the coefficients $B_n(\vartheta)$, one can still ensure that  $\partial_\vartheta e\vert_{\rho} > 0$ so that the system is at least thermodynamically consistent. 

We next consider the compressible Navier--Stokes--Fourier (CNSF) equations for the corresponding state laws,
\begin{align} 
&\partial_{t}\rho+\div(\rho u) =0, \label{eq:00}\\
&\partial_t (\rho u) + \div(\rho u\otimes u)  - \div\mathcal S
+ \nabla P = 0, \label{eq:01}\\
&\partial_t(\rho E) + \div(\rho u E) + \div(Pu) = \div(\mathcal Su) + \div (\kappa \nabla \vartheta)
\label{Energy}
\end{align}
where $E =|u|^2/2 + e$ is the energy with $P = P(\rho, \vartheta)$ and $e = e(\rho, \vartheta)$ respectively stand for the pressure and the (specific) internal energy.

The initial condition are given by
  \begin{equation}
  \label{eq:ini}
  \rho\vert_{t=0} =\rho_0 \ \ \ \ \ (\rho u)\vert_{t=0} =m_0\ \ \ \ \ (\rho E)\vert_{t=0} =\rho_0E_0.
  \end{equation}  
Note that the above initial conditions determine the corresponding value at $t=0$ of the temperature $\vartheta|_{t=0} = \vartheta_0$, provided that $\partial_\vartheta e>0$. 

For simplicity, we take the isotropic stress tensor
\begin{equation}\label{viscous}
	\mathcal{S} = \mu(\nabla u + \nabla u^T) + \lambda\div u\mbox{Id}
\end{equation} 
with $\mu$ and $\lambda$ two constants satisfying the physical constraint $\mu>0$ 
and $\lambda +2\mu/d> 0$. 
In order to be consistent with the second principle of Thermodynamics which implies
the existence of the entropy as a closed differential form in the energy balance, the
following compatibility condition, called ``Maxwell equation" between $P$ and $e$ has to be
satisfied
  \begin{equation}
  \label{max:equ}
  P = \rho^2\frac{\partial e}{\partial\rho} + \vartheta \frac{\partial P}{\partial\vartheta}.
  \end{equation}
  This allows to define the internal energy directly from the pressure law, up to a function only of $\vartheta$, which we take as $0$ for simplicity. Namely for any arbitrary $\rho_c>0$, and by using~\eqref{virial},
  \begin{equation}
    \begin{split}
      e(\rho,\vartheta)&=\int_{\rho_c}^\rho -\frac{1}{\rho^2}\,\left(P(\rho',\vartheta)-\vartheta\,\frac{\partial P}{\partial\vartheta}(\rho',\vartheta)\right)\\
      &=m+\frac{\rho^{\gamma-1}}{\gamma-1}-\sum_{n=0}^N \vartheta^2\,\frac{d}{d\vartheta}(B_n(\vartheta))\,\frac{\rho^n}{n-1}-\vartheta^2\,\frac{d}{d\vartheta}(B_1(\vartheta))\,\log \rho+\vartheta^2\,\frac{d}{d\vartheta}(B_1(\vartheta))\,\frac{1}{\rho}.
\end{split}
\label{statelaw}
  \end{equation}
The specific entropy $s=s(\rho, \vartheta)$ is now also defined up to an additive constant by
  \begin{equation}\label{entrop}
  \frac{\partial s}{\partial \vartheta}\bigg|_{\rho}=\frac{1}{\vartheta}\frac{\partial e}{\partial\vartheta}\bigg|_{\rho}\ \ 
  \ \ \ and \ \ \ \ \ 
  \frac{\partial s}{\partial \rho}\bigg|_{\vartheta}=-\frac{1}{\rho^2}\frac{\partial P}{\partial\vartheta}\bigg|_{\rho}.
  \end{equation}
 If $(\rho,\ \vartheta)$ are smooth and bounded from below away from zero and if the velocity field is smooth, then the total energy balance can formally be replaced by the thermal energy balance
\begin{equation*}
	C_v\rho(\partial_{t}\vartheta + u\cdot\nabla \vartheta) - \div(\kappa(\vartheta)\nabla\vartheta) = \SS:\nabla u - \vartheta \frac{\partial P(\rho, \vartheta)}{\partial \vartheta}\div u
\end{equation*}
where $\cs:\nabla u = \Tr(\cs\nabla u)$.

Furthermore, dividing by $\vartheta$, we arrive at the entropy equation
\begin{equation} \label{ENTROPY}
	\partial_t(\rho s) + \div (\rho s u) - \div\left(\frac{\kappa(\vartheta)\nabla\vartheta}{\vartheta}\right) = \frac{1}{\vartheta}\left(\SS:\nabla u + \frac{\kappa|\nabla\vartheta|^2}{\vartheta}\right).
\end{equation}
We will use both of the two equations~\eqref{Energy} and~\eqref{ENTROPY} involving temperature, at different parts of our argument, together with a third technical formulation derived from~\eqref{Energy}.

We emphasize that, {\em a priori}, the system~\eqref{eq:00}-\eqref{Energy} conserves 
 the total mass 
$$ \int_{\TT^d} \rho (t,\cdot)\, dx = \int_{\TT^d} \rho_0 \, dx.
$$
The total energy of the system, which is the sum of the kinetic and the potential energies, reads
  \begin{equation*}
    \ce(\rho,\vartheta,m) = 
    \int_{\TT^{d}} \rho E(\rho,\vartheta, \rho u) \, dx =
     \int_{\TT^{d}} \left(\frac{|m|^2}{2\rho } + \rho e(\rho, \vartheta)\right)\,dx 
  \end{equation*}
and is also conserved, namely,
$$ \ce(\rho,\vartheta,m)(t)  =   \ce(\rho_0,\vartheta_0,m_0),$$
with $m=\rho u$, where $e$ is obtained from equation~\eqref{max:equ}.

\bigskip

We need several precise assumptions on the various coefficients entering into equations~\eqref{eq:00}-\eqref{Energy} which we now make explicit.

\smallskip

\noindent {\bf Assumption on the conductivity $\kappa(\vartheta)$:}
\begin{align}
	\label{con:ass}
	\kappa_1(\vartheta^{\alpha}+1) \le \kappa(\vartheta) \le  \kappa_2(\vartheta^{\alpha}+1), \ \ \ \ \ \ \ \ \ \  \kappa_3\vartheta^{\alpha-1} \le \kappa'(\vartheta) \le  \kappa_4\vartheta^{\alpha-1}
\end{align}
where $\kappa_1, \kappa_2>0,$ and $\alpha\ge 4$.

\smallskip
\noindent {\bf Assumptions on the pressure law $P$.}
\begin{itemize}
	\item[(1)]
The pressure $P$ given by \eqref{virial} contains a radiative part, namely 
\begin{equation} \label{P2}   
\partial_\vartheta^2B_0 >0\ \hbox{for}\ \vartheta=0;
\end{equation} 
\item[(2)] For $2\le\gamma_{\vartheta}\le \alpha/2$ with some $\alpha\ge 4$
 \begin{equation}\label{P3}
  C^{-1}\vartheta^{\gamma_{\vartheta}-1} \le B_0(\vartheta) \le \vartheta^{\gamma_{\vartheta}-1} , \ \ \ \ 
  C^{-1}\vartheta^{\gamma_{\vartheta}-2} \le B'_0(\vartheta) \le \vartheta^{\gamma_{\vartheta}-2};
  \end{equation}
\item[(3)] We assume 
\begin{equation} \label{P4}
B_1\equiv C_1\hbox{ for some } C_1\in\RR;
\end{equation}
\item[(4)] For $n\ge2$, the coefficients $B_n$ is concave in the sense that
\begin{equation} \label{P5} 
  \frac{d}{d\vartheta}(\vartheta^2B'_n) \le 0;
  \end{equation}
\item[(5)]  We also assume that the following is true for $n\ge0$ and $n\neq1$
\begin{equation}
\label{P6} 
 |\vartheta^3 B'''_n(\vartheta)| + |\vartheta^2 B''_n(\vartheta)|+ |\vartheta B'_n(\vartheta)| + |B_n(\vartheta)| \le C \vartheta^{(\gamma-n)\gamma_{\vartheta}/\gamma-1-\epsilon};
\end{equation}
\item[(6)]There exist some constants $\bar B_n$ and $\bar\alpha< \min(\alpha,2\,\gamma_\vartheta)$,
\begin{equation}
		\label{P6bis}
		|\vartheta^2\, B'_n(\vartheta)| + \vartheta|B_n(\vartheta)| + |\vartheta B_n(\vartheta)-\bar B_{n}|  \le  C\,\vartheta^{\bar\alpha(\gamma-2n)/2\gamma-\epsilon};
\end{equation}
\item[(7)]  Finally we also assume the following property on the entropy $s$
\begin{equation} \label{P7} 
\hbox{ The specific entropy } s \hbox{ is a concave function of } (\rho^{-1}, e).
\end{equation}
\end{itemize}

\begin{remark}
    The above assumption on $s$ ensures that the $C_v$ coefficient is non-negative
    \begin{equation}
    	C_v=\frac{\partial e}{\partial\vartheta}\bigg|_\rho = -\frac{1}{\vartheta^2}\frac{\partial^2 s}{\partial e^2}\bigg|_\rho^{-1} \ge 0
    \end{equation}
where the second equality comes from that $\partial s/\partial e = \vartheta^{-1}$. 
\end{remark}

\begin{remark} Let us comment that the results described for in instance in \cite{FeNo}
are based on a radiative part and a cold pressure part. In the truncated pressure law, 
this corresponds respectively to the terms $\vartheta B_0(\vartheta)$ and 
$\rho^\gamma$. 
\end{remark}

We emphasize that none of the assumptions above require a sign on $B_n(\vartheta)$, except on $B_0$. Hence as claimed, the truncated virial pressure may not be monotone in $\rho$ for some values of $\vartheta$ or $\rho$.

\smallskip

We are now ready to state our main result.
\begin{theorem}
\label{main}
Assume the initial data $\vartheta_0$, $m_0$ and $\rho_0\ge 0$ with $\int_{{\mathbb T}^d} \rho_0 = M_0>0$ satisfy
$$\ce(\rho_0, \vartheta_0, m_0) = \int_{\TT^{d}} \left(\frac{|m_0|^2}{2\rho_0 } + \rho_0 e(\rho_0, \vartheta_0) \right)\,dx <\infty$$
where $m_0=0$ when $\rho_0=0$. Suppose that the pressure state law $P(\rho,\vartheta)$ is given by \eqref{virial} with the assumptions \eqref{P2}--\eqref{P6bis} and assume  \eqref{P7} on the entropy.
 Then there exists a global weak solution $(\rho,u,\vartheta)$ to Compressible Navier--Stokes--Fourier System. More precisely it satisfies \eqref{eq:00}--\eqref{eq:01} with \eqref{viscous} in the distribution sense,
 the following entropy inequality
\begin{equation}
  \partial_t (\rho s) + {\rm div}(\rho s u) 
    - {\rm div} \bigl( \frac{\kappa(\vartheta) \nabla\vartheta}{\vartheta}\bigr) \ge
     \frac{1}{\vartheta} \bigl( S:\nabla u + \frac{\kappa(\vartheta) |\nabla\vartheta|^2 }{\vartheta}\bigr) \label{entropyineq}
\end{equation}
where $s$ is defined by \eqref{entrop} and the energy inequality
\[\int_{\TT^d}\bigl( \rho \frac{[u|^2}{2} + \rho e(\rho,\vartheta)\bigr)(t) \, dx 
  \le \int_{\TT^d}\bigl(\frac{|m_0|^2}{2\rho_0} + \rho_0 e(\rho_0,\vartheta_0)\bigr) \, dx.
    \]
Moreover, we have
$$ u \in L^2(0,T;H^1(\TT^d)), \qquad |m|^2/2\rho \in L^\infty(0,T;L^1(\TT^d))$$
$$ \rho u \in  {\mathcal C}([0,T], L^{2\gamma/(\gamma+2)} (\TT^d) \hbox{ weak }),$$
for any $T>0$, the weak regularity 
$$ \rho \in {\mathcal C}([0,T], L^\gamma(\TT^d) \hbox{ weak }) \cap L^{\gamma+a} ((0,T)\times \TT^d)
    \hbox{ where } 0 < a < 1/d$$
$$\vartheta \in L^\alpha (0,T; L^{\alpha/(1-2/d)} (\TT^d)), \qquad
 \log \vartheta \in L^2(0,T; H^1(\TT^d)),
$$
and the initial conditions satisfied by $(\rho,\rho u, \rho s)$ in a weak sense
$$\rho\vert_{t=0} = \rho_0, \qquad \rho u\vert_{t=0} = m_0, \qquad \rho s\vert_{t=0^+} \ge \rho_0
     s(\rho_0, \vartheta_0).$$
\end{theorem}
We remark that we use the notation $\rho\in {\mathcal C}([0,T], L^\gamma(\TT^d) \hbox{ weak })$ to mean that $\rho$ is weakly continuous in time in $L^\gamma$: for any $t_n\to t$, we have that $\rho(t_n,.)\to \rho(t,.)$ for the weak topology of $L^\gamma(\RR^d)$.

\smallskip

Theorem~\ref{main} is the first result providing global existence of weak solutions for the heat conducting Navier-Stokes equations with a  thermodynamical unstable pressure law depending on the density and the temperature.

The main idea in the proof is to separate the density and momentum equations~\eqref{eq:00}-\eqref{eq:01} from the energy equation~\eqref{Energy}. For a given $\vartheta(t,x)$ satisfying appropriate energy bounds, our assumptions on the pressure law let us use~\cite{BrJaWa1} (see also the introductory paper~\cite{BrJa}) more or less directly. This article focused on the barotropic system, namely~\eqref{eq:00}-\eqref{eq:01}, but with pressure laws that are inhomogeneous in time and space. It is thus a good tool for the task of obtaining existence of $\rho$ and $u$ for a given $\vartheta$. 

We also need to obtain existence of some $\vartheta$ solving~\eqref{Energy} for a given $\rho$ and $u$, again with appropriate energy bounds. This does not seem to fit in any classical framework of non-linear parabolic equations and therefore requires careful approach. We use a different formulation, that is loosely based on~\eqref{Energy} (and formally equivalent when all quantities are smooth). We also need a proper approximated equation to resolve a potential degeneracy where $\vartheta$ is close to $0$. This finally allows us to obtain a global, weak solution to our variant formulation to~\eqref{Energy}. We do not have strong enough bounds to recover~\eqref{Energy} rigorously from that but it is enough to obtain an inequality in the entropy formulation~\eqref{entropyineq} together with the opposite inequality in the propagation of the total energy (as can be surmised from the formulation in Theorem~\ref{main}). 

The last step in the proof is obviously to conclude the fixed point argument, through the Leray-Schauder theorem. This is a rather short but very challenging step. The issue is that we cannot yet recover the usual energy estimate: Before we do obtain a fixed point, the piece of the energy that we obtain from the existence on~\eqref{eq:00}-\eqref{eq:01} does not fit with the piece of the energy that we obtain from~\eqref{Energy}. This is where the exact formulation of the Leray-Schauder theorem is critical and must be combined with the precise choice we have made of the decomposition.

\section{Previous result concerning the compressible  Navier-Stokes-Fourier system.}  In every previous work concerning the global existence of weak solutions, the viscous stress tensor is assumed to be isotropic
\[
{\mathcal  S}= \mu\,(\nabla u+\nabla u^T)+\lambda\,\div u\, {\rm Id},
\]
with coefficients $\mu,\;\lambda$ either constant or depending only on $\vartheta$.  Concerning the pressure state law, we can cite the two following assumptions:

\medskip

\noindent {1) \it The pressure law as a monotone perturbation of the barotropic case.} It is due  to {\sc E. Feireisl} who considered  pressure laws under the form
\[
P(\rho,\vartheta) = P_c(\rho) + \vartheta P_\vartheta(\rho),
\]
where
\begin{equation}\begin{split}
& P_c(0)=0, \quad P'_c(\rho) \ge a_1 \rho^{\gamma-1} - b \hbox{ for } \rho >0,\\
& P_c(\rho) \le a_2 \rho^\gamma + b \hbox{ for all } \rho \ge 0,\\
& P_\vartheta(0) = 0, \qquad P'_\vartheta(\rho) \ge 0 \hbox{ for all } \rho \ge 0,\\
& P_\vartheta(\rho) \le c(1+\rho^\beta),
\end{split}\label{pressureheat}
\end{equation}
and
\[
\gamma >d/2, \qquad \beta < \frac{\gamma}{2} \hbox{ for } d=2, \quad
    \beta=  \frac{\gamma}{3} \hbox{ for } d=3
\]
with constants $a_1>0$, $a_2$, $b$ and $P_c$, $P_\vartheta$ in
${\mathcal  C}[0,+\infty) \cap {\mathcal  C}^1(0,+\infty)$.
   In agreement with Maxwell law and the entropy definition, it implies the following form
on the internal energy 
\[
e(\rho,\vartheta) = \int_{ \rho_\star}^\rho \frac{P_c(s)}{s^2} ds 
                               + Q(\vartheta),
\]
where $Q'(\vartheta)= C_v(\vartheta)$ (specific heat at constant volume). The entropy is given by
\[
s(\rho,\vartheta) = \int_{ \rho_\star}^\vartheta \frac{C_v(s)}{s} ds -  H_\vartheta(\rho),
\]
where $H_\vartheta(\rho)$ is the thermal pressure potential given through 
$$\displaystyle H_\vartheta(\rho) = \int_{ \rho_\star}^\rho {P_\vartheta(s)}/{s^2} ds.$$
   The heat conductivity coefficient $\kappa$ is assumed to satisfy
  \[
\kappa_1(\vartheta^\alpha + 1)  \le \kappa(\vartheta) \le \kappa_2(\vartheta^\alpha+1) 
     \hbox{ for all } \vartheta \ge 0,
\]
 with constants $\kappa_1>0$ and $\alpha\ge 2$.
  The thermal energy $Q=Q(\vartheta) = \int_0^\vartheta C_v(z) dz$ has not yet been determined and is assumed to satisfy $\int_{z\in [0,+\infty)} C_v(z)>0$ and $C_v(\vartheta) \le c(1+\vartheta^{\alpha/2-1})$.
    Because the energy and pressure satisfy
$$\frac{\partial e(\rho,\vartheta)}{\partial\vartheta} >0, \qquad
   \frac{\partial P(\rho,\vartheta)}{\partial\rho} >0
$$
the estimate on $H_\vartheta$  gives a control on $\rho^\gamma$ in $L^\infty(0,T;L^1(\Omega))$
and through the entropy equation a control on $\vartheta$ in $L^2(0,T; L^6(\Omega))$ in dimension $3$ and in $L^2(0,T;L^p(\Omega))$ for all $p<+\infty$ in dimension $2$.

\medskip

Because the entropy estimates does not provide an $H^1_x$ bound on $u$, {\sc E. Feireisl}  
combines it with a direct energy estimate (see below).
   Therefore one obtains the exact equivalent of estimates as in the barotropic case
\begin{equation}\begin{split}
&\sup_t \int_\Omega \rho\,|u|^2\,dx\leq C+{\mathcal E} (\rho_0,u_0),\\
& \sup_t \int_\Omega \rho^\gamma\,dx\leq C,\\
&\int_0^T\int_\Omega |\nabla u|^2\,dx\leq C.
\end{split}\label{energyestimates}
\end{equation}
in  this temperature dependent case. Using such information, he may then prove the extra integrability
$$\int_0^T \int_\Omega \rho^{\gamma+a} \, dx dt \le C(T,\ce(\rho_0,u_0))
$$
for $0<a<\min (1/d, 2d/\gamma -1)$. We will give more details later-on for such estimate for
the truncated virial pressure law.


 \bigskip

 \noindent {2) \it  Self-similar pressure laws with large radiative contribution.} It is due to {\sc E. Feireisl} and {\sc A.~Novotny }  who consider   pressure laws exhibiting  both coercivity of type $\rho^\gamma$ and $\vartheta^4$ for large densities and temperatures namely 
\[
P(\rho,\vartheta) = \vartheta^{\gamma/(\gamma-1)} Q(\frac{\rho}{\vartheta^{1/(\gamma-1)}}) +\frac{a}{3}  \vartheta^4 \hbox{ with }a>0, \quad \gamma >3/2,
\]
with 
\[
Q\in {\mathcal  C}^1([0,+\infty)), \qquad Q(0)=0, \qquad Q'(Z) >0 \hbox{ for all } Z\ge 0,
\]
and
\[
\lim_{Z\to +\infty} \frac{Q(Z)}{Z^\gamma} = Q_\infty >0.
\]
In agreement to Maxwell law and the definition of entropy, it implies the following
form on the internal energy and the entropy
\[
e(\rho,\vartheta) = \frac{1}{(\gamma-1)} \frac{\vartheta^{\gamma/(\gamma-1)}}{\rho}
     Q(\frac{\rho}{\vartheta^{1/(\gamma-1)}}) + a \frac{\vartheta^4}{\rho},
\]
and 
\[
s(\rho,\vartheta) = S\left(\frac{\rho}{\vartheta^{1/(\gamma-1)}}\right) + \frac{4a}{3}\frac{\vartheta^3}{\rho}.
\]
They
impose
\[
0 < -S'(Z) =\frac{1}{\gamma-1} \frac{\gamma Q(Z) - Q'(Z)Z}{Z} < c  < +\infty \hbox{ for all } Z>0,
\]
with $\lim_{Z\to +\infty} S(Z) = 0$ so that thermodynamical stability holds.
Therefore the energy provides uniform bounds in $L^\infty_t L^1_x$ for $\vartheta^4$ and 
 $\rho^\gamma$.
   One assumes in this case that the viscosities and heat conductivity satisfy
 \[
\mu, \lambda \in {\mathcal  C}^1([0,+\infty)) \hbox{ are Lipschitz with } \quad
     \mu\, (1+\vartheta) \le \mu(\vartheta), \quad 0\le \lambda(\vartheta), \quad  \mu_0>0,
 \]
and
\[
\kappa \in {\mathcal  C}^1([0,+\infty), \qquad 
    \kappa_0 (1+\vartheta^3) \le \kappa(\vartheta) \le \kappa_1 (1+\vartheta^3), 
    \qquad 0<\kappa_0\le \kappa_1.
 \]
Almost everywhere convergence of the temperature is obtained using
the radiation term. Extra integrability on $P(\rho,\vartheta)$ can be derived just as in the  barotropic case. 
Finally the same procedure as in the barotropic case is followed to have compactness on the density, relying heavily on the monotonicity of the pressure ${\partial P(\rho,\vartheta)}/{\partial\rho}>0$. This gives global existence of weak solutions (in a the same sense that we precise later). Remark the term 
$a\vartheta^4/3$ in the pressure law can help to get compactness in space and time using 
commutation between strictly convex function and weak convergence.

\bigskip

With respect to these previous works, we focus here, as in the barotropic case, in removing the assumption of monotonicity on the pressure law with respect to the density, considering the truncated virial pressure law on which we do not want to assume too restrictive assumptions, namely, a pressure law \eqref{virial} with the assumptions \eqref{P2}--\eqref{P7}.

%
\section{The direct entropy estimate}
\subsection{A formal entropy bound}
We explain here the general framework for our result on the Navier--Stokes--Fourier system. The estimates here closely follow the ones pioneered  by {\sc P.--L. Lions}, and {\sc E. Feireisl} and {\sc A. Novotny}. With respect to the previous discussion, we only present them here in a more general context as in particular we will not need the monotonicity of $P$.  

If one removes the monotonicity assumption on $P$ then thermodynamic stability does not hold anymore. 
Following {\sc  P.--L. Lions}, it is however possible to obtain the entropy dissipation 
 estimate directly by integrating the entropy equation
\[
\int_0^t\int_\Omega \Bigl(\mu\,\frac{|\nabla u|^2}{\vartheta} 
     +  \kappa{(\vartheta)}\frac{ |\nabla\vartheta|^2}{\vartheta^2}\Bigr)\,dx\,dt \leq C\,\int_\Omega \rho\,s(t,x)\,dx.
\]
Therefore the entropy bound dissipation holds under the general assumption that there exists $C$ s.t.
\begin{equation}
s(\rho,\vartheta)\leq C\,e(\rho,\vartheta)+\frac{C}{\rho}.\label{boundentropyenergy}
\end{equation}
Recall that
\[
e=m(\vartheta)+\int_{\rho^\star}^\rho \frac{(P(\rho',\vartheta) - \vartheta\partial_{\vartheta}P(\rho',\vartheta))}{\rho'^2}\,d\rho',
\]
and
\[
\partial_\rho s= 
-  \frac{\partial_\vartheta P}{\rho^2}.
\]
We also have that $\partial_\vartheta s =\partial_\vartheta e / \vartheta$, therefore as long as $m(\vartheta)\geq 0$ with 
\[
\int_{\vartheta^*}^\vartheta \frac{m'(s)}{s}\,ds \le C(1+m(\vartheta)),
\]
and
\[
 -\int_{\rho^*}^\rho \frac{\partial_\vartheta P}{\rho'^2}\,d\rho' \le
  C+C\,\int_{\rho^*}^\rho \frac{P -\vartheta \partial_\vartheta P}{\rho'^2}\,d\rho',
\]
for some $C>0$  then \eqref{boundentropyenergy} is automatically satisfied  
and one obtains the entropy bound dissipation.
Moreover if $e(\vartheta,\rho)\geq \rho^{\gamma-1}/C$ then one also has that 
$\rho\in L^\infty_t\,L^\gamma_x$.

  Assuming now that
\[
\kappa_1\,(\vartheta^\alpha+1)\leq \kappa(\vartheta)\leq \kappa_2\,(\vartheta^\alpha+1),
\]
with $\alpha \ge 2$, one deduces from the entropy estimate that
\[
\int_0^T\int_{\TT^d} (\vartheta^\alpha+1)\,|\nabla\vartheta|\,dxdt<\infty,
\]
showing that $\log \vartheta\in L^2_t H^1_x$ and $\vartheta^{\alpha/2}\in L^2_t\,H^1_x$ or 
by Sobolev embedding $\vartheta\in L^{\alpha}_t\,L^{\alpha/(1-2/d)}_x$ for $d\geq 3$.

   By a H\"older estimate, it is also possible to obtain a Sobolev-like, $L^{p_1}_t W^{1,p_2}_x$, bound on $u$
\[\begin{split}
&\int_0^T\left(\int_\Omega |\nabla u|^{p_2}\,dx\right)^{p_1/p_2}\,dt
\leq \left(\int_0^T\int_\Omega \frac{|\nabla u|^2}{\vartheta}\,dx\,dt\right)^{p_1/2}\\
&\qquad\times\left(\int_0^T\left(\int_\Omega \vartheta^{p_2/(2-p_2)}\,dx\right)^{p_1(2-p_2)/(p_2(2-p_1))}dt\right)^{(2-p_1)/2}<\infty,
\end{split}\] 
provided that $p_2/(2-p_2)=\alpha/(1-2/d)$ and $p_1/(2-p_1)=\alpha$ or
\begin{equation}
p_1=\frac{2\,\alpha}{1+\alpha},\quad p_2=\frac{2\,\alpha\,d}{d\,(\alpha+1)-2}.\label{sobolevexputemp}
\end{equation}
Unfortunately this Sobolev estimate does not allow to derive the gain of integrability on the density 
as usually. Actually one requires a $L^2_t H^1_x$ estimate on $u$ (the critical point is in fact the $L^2_t$ with value in some Sobolev in $x$).
   Instead one can easily extend the argument by {\sc E. Feireisl} and {\sc A. Novotny}: For any $\phi(\rho)$, one can write
\[
\frac{1}{2}\frac{d}{dt} \int_\Omega \rho |u|^2 + \frac{d}{dt} \int _\Omega\phi(\rho)
    + \int_\Omega {\mathcal  S}:\nabla u 
    = \int_\Omega (P(\vartheta,\rho)-\phi'(\rho)\,\rho+\phi(\rho))\,{\rm div u}.
\]
This leads to the assumption that there exists some $\phi$ s.t.
\begin{equation}
\begin{aligned}
& C^{-1} \rho^\gamma - C   \le \phi(\rho) \le C \rho^\gamma + C, \\
&  \left|P(\vartheta,\rho)-\phi'(\rho)\,\rho+\phi(\rho)\right|\leq C
\Bigl( \rho^{\beta_1}+\vartheta^{\beta_2}+\sqrt{\rho e(\vartheta,\rho)}\Bigr),
\label{pressureminusbarotrope}
\end{aligned}
\end{equation}
 with
\begin{equation}
\beta_1\leq \frac\gamma 2,\quad \beta_2\leq \frac\alpha 2. \label{assumptionbetas}
\end{equation}
Indeed, with \eqref{pressureminusbarotrope}, one has
\[
\int_0^T\int_\Omega {\mathcal  S}:\nabla u\,dx\,dt\leq {\mathcal E} (\rho_0,\vartheta_0,m_0)
    +C\,\int_0^T\int_\Omega \Bigl(\rho^{\beta_1}+\vartheta^{\beta_2}+\sqrt{\rho e(\vartheta,\rho)}\Bigr)\,|\div u|\,dx\,dt.
\] 
Using that ${\mathcal S}$ is Newtonian,  this leads to
\begin{equation}
\int_0^T\int_\Omega |\nabla u|^2\,dx\,dt\leq 
C\, {\mathcal E}(\rho_0,\vartheta_0,m_0)+C\,\|\nabla u\|_{L^2_{t,x}}\,\|\rho^{\beta_1}+\vartheta^{\beta_2}\|_{L^2_{t,x}},\label{H1withtemp}
\end{equation}
and the desired $H^1$ bound follows from \eqref{assumptionbetas}. It is now possible to follow the same steps to obtain an equivalent of  the extra integrability on the density if $\gamma>d/2$
\begin{equation}
\int_0^T \int_\Omega \rho^{\gamma+a}\,dx\,dt\leq C(T,{\mathcal E}(\rho_0,\vartheta_0, m_0)),\quad\mbox{for all}\ a< 1/d. \label{gainintegrabilitytemp}
\end{equation}
Note here that the assumptions \eqref{pressureminusbarotrope}-\eqref{assumptionbetas} are likely not optimal. They nevertheless already cover the truncated virial law we consider here.

\subsection{The assumptions on the pressure law to get the above estimates}
For convenience, we repeat here all the assumptions presented above and will show that
the truncated virial pressure laws satisfy them under the assumptions related to the coefficients
$B_n$. To derive the important estimates mentioned in the sections above, the pressure law $P(\rho,\vartheta)$ has to be a positive pressure law satisfying the following
properties: For some $C>0$ and $\gamma>d$
\begin{equation}
\left\{\begin{aligned}
&P(\rho,\vartheta) \hbox{ such that } -\int_{\rho^*}^\rho \frac{\partial_\vartheta P}{\rho'^2}\,d\rho' \le
  C+C\,\int_{\rho^*}^\rho \frac{P -\vartheta \partial_\vartheta P}{\rho'^2}\,d\rho' , \\
& e(\rho,\vartheta)= m(\vartheta) + \int_{\rho^\star}^\rho \frac{(P(\rho',\vartheta) -
     \vartheta \partial_\vartheta P(\rho',\vartheta))}{\rho'^2}\,d\rho'\geq 
      m(\vartheta) +
     \frac{\rho^{\gamma-1}}{C} +\frac{\vartheta^{\gamma_\vartheta}}{C\,\rho}.\\
&\quad   \hbox{ with }  m(\vartheta)\ge 0 \\
& \hskip1cm  \hbox{ and }  \int_{\vartheta^*}^\vartheta \frac{m'(s)}{s}\,ds \le C(1+m(\vartheta))\leq C\,(1+\vartheta^{\alpha\,(\gamma+a-1)/2(\gamma+a)}), \\
&\kappa_1\,(\vartheta^\alpha+1)\leq \kappa(\vartheta)\leq \kappa_2\,(\vartheta^\alpha+1),\quad \mu,\;\lambda \ \mbox{constant and } \alpha\ge 4,\\
& C^{-1} \rho^\gamma- C\leq \phi(\rho)\leq C\,\rho^\gamma + C, \\ 
&|P(\rho,\vartheta)-\phi'(\rho)\,\rho+\phi(\rho)|
  \leq C\,\rho^{\beta_1}+C\,\vartheta^{\beta_2}+C\,\sqrt{\rho e(\rho,\vartheta)},  \\
  & |\partial_\vartheta P(\rho,\vartheta)| \le C \rho^{\beta_3} + C \vartheta^{\beta_4}
\end{aligned}\right.\label{assumptionswithtemp}
\end{equation}
for
\begin{equation}\left\{\begin{aligned}
&\beta_1\leq \frac{\gamma}{2},\quad \beta_2< \frac{\alpha}{2},\\
& \beta_3 < \frac{\gamma+ a+1}{2}, \qquad \beta_4 < \frac{\alpha}{2},
\\
&\frac{2}{d}\,\mu+\lambda>0, \quad \gamma_{\vartheta} \ge 2\\
\end{aligned}\right.\label{coefftemp}
\end{equation}
where we recall that $a<\min\left(1/d, 2\gamma/d -1\right)=1/d$ since $\gamma>d$ here.
  We also assume that the specific heat is positive (as is necessary for the physics) {\em i.e.}
\begin{equation}
C_v=\partial_\vartheta e(\rho,\vartheta)>0,\quad \forall \rho,\;\vartheta,\label{specificheat>0}
\end{equation}
and that the pressure contains a radiative part
\begin{equation}
\partial_\vartheta^2 P(\rho=0,\vartheta)>0.\label{radiative}
\end{equation}
We do not need to impose monotony on $P$ and it is enough that
\begin{equation}
\begin{split}
&\bigl|\partial_\rho P(\rho,\vartheta)\bigr|\leq C\,\rho^{\gamma-1}+C\,\vartheta^{\gamma_\vartheta}
\qquad \hbox{ with }  2 \le  \gamma_\vartheta < {\alpha}/{2}.\\
\end{split}\label{nonmonotonewithtemp}
\end{equation}
Finally the initial data has to satisfy
\begin{equation}\label{initemp1}
\begin{aligned}
&\rho_0 \in L^\gamma(\Pi^d), \quad
\vartheta_0 \in L^{\gamma_{\vartheta}}(\Pi^d)  \\
& \hskip2cm \hbox{ with } \rho_0\ge 0,  \quad 
      \vartheta_0 >0 \hbox{ in } \Pi^d   \qquad \hbox{ and }
       \int_{\Pi^d} \rho_0 = M_0>0,  
      \end{aligned}
      \end{equation}
 and
 \begin{equation}\label{initemp2}
      {\mathcal E}_0 = \int_{\Pi^d} \Bigl(\frac{1}{2}\frac{|(\rho u)_0|^2}{\rho_0} + \rho_0 e(\rho_0,\vartheta_0)\Bigr) <+\infty.
\end{equation}

\bigskip

\subsection{The truncated virial pressure law satisfies the properties needed for the estimates above}
  The truncated pressure law mentioned in the introduction\[
P(\rho,\vartheta)=\rho^\gamma
                               +\vartheta\,\sum_{n=0}^{N} B_n(\vartheta)\,\rho^n 
\]
with $\gamma>2N\geq 4$ satisfy the assumptions described before. Choosing $m=constant$ for simplicity in this example, this leads to 
\[
e(\rho,\vartheta)=m+\frac{\rho^{\gamma-1}}{\gamma-1}
   -\sum_{n\ge 2}^N \vartheta^2\,B_n'(\vartheta)\,\frac{\rho^{n-1}}{n-1}
   - \vartheta^2 B'_1(\vartheta) \log \rho
   + \vartheta^2 B'_0(\vartheta) \frac{1}{\rho} ,
\]

\noindent For simplicity,  let us assume that $B_1= constant$, which is the normal virial assumption,  so that this term vanishes. 
The entropy reads
\[
s(\rho,\vartheta)=-\sum_{n\ge 2}^N (\vartheta\,B_n'(\vartheta)+B_n(\vartheta))\,\frac{\rho^{n-1}}{n-1}
  +B_1\,\log \rho + (\vartheta B_0'(\vartheta)+ B_0(\vartheta)) \frac{1}{\rho}.
\]

\medskip

\noindent {\rm 1)} We assume that the pressure contains a radiative part, namely that $B_0$ is convex in $\vartheta$ with $ C^{-1}\,\vartheta^{\gamma_\vartheta -1}\leq B_0(\vartheta) \leq \vartheta^{\gamma_\vartheta-1}$ and $ C^{-1}\,\vartheta^{\gamma_\vartheta-2}\leq B_0'(\vartheta) \leq \vartheta^{\gamma_\vartheta-2}$ where $2\le \gamma_\vartheta \le \alpha/2$. This already satisfies \eqref{radiative}. 

\medskip

\noindent {\rm 2)} For $n\geq 2$ , the coefficients $B_n$ can have any sign but we require a concavity assumption: 
\[
\frac{d}{d\vartheta}(\vartheta^2\,B_n') \leq 0.
\]
 This ensures, with the assumption on $B_0$, that the specific heat $C_v = \partial_\vartheta e (\rho, \vartheta)$ satisfies {\rm \eqref{specificheat>0}}. This is again a classical assumption for the virial.  Note that it would be enough to ask this concavity of some of them and moreover that this is automatically satisfied if $B_n\sim -\vartheta^\nu$, that is precisely for the coefficients contributing to the non-monotony of $P$ in $\rho$.

\medskip

\noindent {\rm 3)} We also require some specific bounds on the $B_n$ namely that there exist $\bar B_n$ and $\eps>0$ s.t.
\begin{eqnarray}
&& \displaystyle 
\nonumber |\vartheta\, B_n'(\vartheta )|+|B_n(\vartheta)|\leq C\,\vartheta^{
    \frac{\gamma-n}{\gamma}\,\gamma_\vartheta-1-\eps},\\
&&  \displaystyle 
|B_n(\vartheta)|+|\vartheta\, B_n(\vartheta )-\bar B_n|\leq C \vartheta^{\frac{\alpha}{2}\left(1-\frac{2n}{\gamma}\right)-\eps}.\label{assumeBn}
\end{eqnarray}
First of all this gives us a bound from below on $e$
\[\begin{split}
e(\rho,\vartheta)&\geq m+\frac{\rho^{\gamma-1}}{\gamma-1}+C^{-1}\,\frac{\vartheta^{\gamma_\vartheta}}{\rho}-\sum_{n=2}^N \vartheta\,\vartheta^{\frac{\gamma-n}{\gamma}\,\gamma_\vartheta-1-\eps} \,\frac{\rho^{n-1}}{n-1}\\
&\geq m+\frac{\rho^{\gamma-1}}{\gamma-1}+C^{-1}\,\frac{\vartheta^{\gamma_\vartheta}}{\rho}-C\,\sum_{n=2}^N \left(\rho^{\gamma-1-\eps'} +\frac{\vartheta^{\gamma_\vartheta-\eps'}}{\rho}\right),
\end{split}\]
by Young's inequality, so that this implies \eqref{assumptionswithtemp}$_{2}$. Assumption \eqref{nonmonotonewithtemp} is proved with an identical calculation. The same calculation also proves Assumption \eqref{assumptionswithtemp}$_{1}$ by showing that $s\leq C\,(\rho^{\gamma-1}+\vartheta^{\gamma_\vartheta}+1)$.

\medskip

\noindent {\rm 4)}  Then choosing 
\begin{equation}\label{phi}
\phi(\rho)=\frac{\rho^\gamma}{\gamma-1}+\sum_{ 2\leq n\leq N} \bar B_n\,\frac{\rho^n}{n-1}
      + \bar B_0,
\end{equation} 
and using again the second part of \eqref{assumeBn}
 we have that
\[\begin{split}
|P-\phi'(\rho)\,\rho+\phi(\rho)|&\leq C\,\sum_{n\leq N} |\vartheta\, B_n(\vartheta )-\bar B_n|\,\rho^n\leq C\,\sum_{n\leq N} \vartheta^{\frac{\alpha}{2}\left(1-\frac{2n}{\gamma}\right)-\varepsilon }\,\rho^n\\
&\leq C\,N\,(\rho^{\gamma/2}+\vartheta^{\alpha/2-\eps}),
\end{split}\]
still by Young's inequality. This yields the wanted estimate with the right inequalities on $\beta_1$ and $\beta_2$.
 The same calculation also proves that $|\partial_\vartheta P|\leq C\,(\rho^{\gamma/2+\eps'}+\vartheta^{\alpha/2-\eps'})$ with required assumptions on $\beta_3,\;\beta_4$.

\medskip

Note for a fixed $\vartheta$ then $P(\rho,\vartheta)$ is indeed increasing with respect to $\rho$ after a critical $\bar\rho_\vartheta$ which depends on $\vartheta$ and can be arbitrarily large where $\vartheta>>1$. This is the reason why $P$ does not satisfy any of the classical monotonicity assumption and why our new approach is needed.
  Our pressure law has two important parts: the radiative term (corresponding to $n=0$) to get compactness on the temperature and the $\phi(\rho)$ term to get compactness on the density 
in time and space.

\begin{remark}
In our work, the viscosity coefficients $\mu,\;\lambda$ are independent of the temperature $\vartheta$. Instead several models use temperature dependent coefficients $\mu(\vartheta),\;\lambda(\vartheta)$. To handle that case, the proof given below would have to be modified; the compactness of the temperature would have to be established first following what has been done in previous work
for monotone pressure laws in previous works.
\end{remark}

%
\section{The new strategy to get global existence of weak solutions}
The main difficulty is the construction of regular enough solutions of some approximate system that will allow us to derive our key a priori estimates and pass to the limit. We obtain solutions of the approximate system through a fixed point argument that strongly relies on our recent paper \cite{BrJaWa1}.

\smallskip

\noindent More precisely, we consider the following {\bf First Step:}  
We start with a prescribed temperature $\vartheta$ so that $p(\rho,\vartheta(t,x))$ satisfies the assumptions in \cite{BrJaWa1} for an heterogeneous pressure $p(\rho,t,x)$. 
This provides a map $\vartheta \mapsto (\rho,u)$ with $(\rho,u)$ a global weak solution of
\begin{equation}
  \partial_t \rho + {\rm div}(\rho u) = 0,\label{continuity}
  \end{equation}
and 
\begin{equation}
\partial_t (\rho u) + {\rm div} (\rho u \otimes u) - \mu \Delta u - (\lambda+\mu) \nabla {\rm div} u
+ \nabla (p(\rho,\vartheta(t,x))) = 0.\label{momentum}
\end{equation}

\smallskip 
\noindent The novelty in the present paper is the {\bf Second Step}:
The construction of the associated  temperature through the energy equation corresponding to the pressure without the barotropic part $\rho^{\gamma-1}/(\gamma-1)$
namely the one corresponding to
$$\widetilde P (\rho,\vartheta) = \vartheta \sum_{n=0}^N B_n (\vartheta) \rho^n.$$
To do so, we rewrite the energy equation
\begin{equation}
  \partial_t (\rho \widetilde e) + {\rm div} (\rho \widetilde e u) + \widetilde P {\rm div} u =
          {\mathcal S}: \nabla u + {\rm div} (\kappa(\vartheta)\nabla\vartheta)
          \label{modifiedenergy}
          \end{equation}
      in terms of a quasi-linear parabolic equation on $\widetilde g= \rho \tilde e$.

      Our goal is then to complete the fixed point argument by solving this equation for a fixed density $\rho$ and velocity field $u$. This is a non-trivial problem as the equation is singular and it requires several extra steps:
\begin{itemize}
      \item First of all, we need to regularize $\rho$ and $u$ in space and time, be far from vacuum for the density and remove the singularity in $\tilde e$ near the $0$ temperature using a parameter $\varepsilon$. That will allow us to use in a first step the classical result \cite{LaSoUr} by O.A. Ladyzenskaya, V.A. Solonnikov, N.N. Uraltceva and  get existence (\cite{LaSoUr}) of some classical solution $\widetilde g_\eps$ for the regularized equation.
\item In a second step, we may pass to the limit $\varepsilon \to 0$ to obtain the actual solution $\widetilde g$, using the expected a priori estimates for classical solutions.
\item The third step  consists in recovering the temperature $\vartheta$
such that $\rho \widetilde e (\rho, \vartheta) = \widetilde g$ using an implicit function 
procedure thanks to the key property $\partial_\vartheta \widetilde e >0$ and the fact 
that $\vartheta$ is more regular.
\item The last step
consists in deriving uniform estimates on $\vartheta$ by transforming estimates from step 1 and making sure that there are uniform in the various regularizing parameters. This uniform
estimate is obtained through the entropy equation derived from the
energy equation as explained earlier.
\end{itemize}
Once this is done, we obtain a map on the temperature $\vartheta$: From an initial $\vartheta_i$, we obtain $(\rho,\,u)$ solving \eqref{continuity}-\eqref{momentum}. We then obtain the ``new'' temperature $\vartheta$ that solves \eqref{modifiedenergy} for those $\rho$ and $u$. 

\smallskip

\noindent The {\bf Third and Last Step} is then to get a fixed point on the temperature, for example through Schauder theorem, we need to obtain some compactness on the map. This turns out to be rather straightforward: if  $\vartheta_i$ is bounded in some appropriate Sobolev space, then  $\log \vartheta$ belongs to some $H^1$ and we can derive compactness in space and time using the radiative part in the pressure law. 
\section{First step: Obtaining $\rho$ and $u$ given $\vartheta$}
The goal of this section is to obtain existence of appropriate solutions $\rho$ and $u$ if we already know the temperature $\vartheta$. This will form the first step in our final fixed point argument.

This step heavily relies on the existence result already obtained in \cite{BrJaWa1} to construct
$(\rho,u)$ solution of the compressible Navier-Stokes equations with an heterogeneous pressure
law $P(t,x,\rho)$ with explicit dependence on time and position.
\subsection{Recalling the main result from \cite{BrJaWa1}}
The result in \cite{BrJaWa1} requires the following assumptions with 
$\gamma > 3d/(d+2)$:
\begin{itemize}
\item[(P1)] There exists $q>2$, $0\le \overline \gamma \le \gamma/2$ and a smooth function $P_0$ such that
$$|P(t,x,s) - P_0(t,x,s)| \le C R(t,x) + C s^{\overline \gamma}  
\hbox{ for } R \in L^q ([0,T]\times  T^d)$$
\item[(P2)] There exists $p <\gamma +2\gamma/d - 1$ and $q>2$ with $\vartheta_1 \in L^q([0,T]\times T^d)$, such that
$$ C^{-1} s^\gamma - \vartheta_1(t,x) \le P_0(t,x,s) \le C s^p + \vartheta_1(t,x).$$
\item[(P3)] There exists $p < \gamma + 2\gamma/d -1$ and $\vartheta_2 \in L^q([0,T]\times T^d)$ with
$q>1$ such that
$$|\partial_t P_0(t,x,s)| \le C s^p + \vartheta_2(t,x).$$
\item[(P4)]
$|\nabla_x P_0(t,x,s)| \le C s^{\gamma/2} + \vartheta_3(t,x)$
for $\vartheta_3 \in L^2(0,T; L^{2d/(d+2)}(T^d))$.
\end{itemize}
and the following one for the propagation of compactness on the density:
\begin{itemize} 
\item[(P5)] The pressure $P$ is locally Lipschitz in the sense that
  \[\begin{split}
  |P(t,x,z) - P(t,y,w)| \le Q(t,x,y,z,w) + &C \Bigl((z^{\gamma-1} + w^{\gamma-1})\\
&    + (\widetilde P(t,x) +\widetilde P(t,y))\Bigr)\, |z-w|
  \end{split}
  \]
for some $\widetilde P \in L^{s_0}([0,T]\times T^d)$ with $s_0>1$. Moreover for any sequence $\rho_k(t,x)$ uniformly bounded in $L^\infty([0,\ T],\ L^\gamma(\Pi^d))$ then
$Q(t,x,y,\rho_k(t,x),\rho_k(t,y))$ is uniformly bounded in $L^{s_1} ([0,T] \times T^{2d})$ for some $s_1>1$.
\item[(P6)] For any sequence $\rho_k(t,x)$ uniformly bounded in $L^\infty([0,\ T],\ L^\gamma(\Pi^d))$, the functions $Q,\widetilde P$ satisfy that for some $r_h \to 0$ as $h\to 0$
  \[\begin{split}
  r_h=\sup_k\frac{1}{\|K_H\|_{L^1}} \int_0^T \int_{T^{2d}}
 &K_h(x-y) \bigl(|\widetilde P(t,x) - \widetilde P(t,y)|^{s_0} \\
                &+|Q(t,x,y,\rho_k(t,x),\rho_k(t,y))|^{s_1} 
               \bigr)\, dx dy dt,
  \end{split}
  \]
  where
  \[
K_h(x)=\frac{1}{(h+|x|)^d},\qquad\mbox{for}\quad |x|\leq \frac{1}{4},
\]
with $K_h$ smooth in $\Pi^d\setminus B(0,1/4)$ and with compact support in $\Pi^d\setminus B(0,1/3)$. 
\end{itemize}

\bigskip

We are now ready to recall the main result from \cite{BrJaWa1}
\begin{theorem}
\label{mainprevious}
Assume the initial data $m_0$ and $\rho_0\ge 0$ with $\int_{{\mathbb T}^d} \rho_0 = M_0>0$ satisfy
$$\ce(\rho_0, m_0) = \int_{\TT^{d}} \left(\frac{|m_0|^2}{2\rho_0} + \rho_0 e_P(0,x,\rho_0)\right)\,dx <\infty,$$
where
\[
e_P(t,x,\rho) = \int_{\rho_{ref}}^\rho \frac{P(t, x, s)}{s^2}\,ds
\]
with $m_0=0$ when $\rho_0=0$. Suppose that the pressure $P$ is given by \eqref{virial} with properties \eqref{P2}--\eqref{P6bis}. Then for any $T>0$ there exists a global weak solution to Compressible Navier--Stokes System
 \eqref{eq:00}--\eqref{eq:01} with the strain tensor \eqref{viscous}. Namely it satisfies the equations in a Distribution sense, the following bounds
$$ u \in L^2(0,T;H^1(\TT^d)), \qquad |m|^2/2\rho \in L^\infty(0,T;L^1(\TT^d))$$
$$ \rho \in {\mathcal C}([0,T], L^\gamma(\TT^d) \hbox{ weak }) \cap L^p((0,T)\times \TT^d)
    \hbox{ where } 0 < p < \gamma (d+2)/d -1$$
and the initial conditions in a weak sense with the heterogeneous pressure state law $P$ satisfying the energy inequality
\[\begin{split}
      &\int_{\TT^d} \mathcal{E}_0(\rho,u)\,dx+\int_0^t\int_{\TT^d} \cs:\,\nabla_x u\,dx\,ds\leq {\mathcal E}(\rho_0,u_0)\\
      &\qquad+\,\int_0^t\int_{\TT^d} \div_x u(s,x)\,(P(s,x,\rho(s,x))-P_0(s,x,\rho(s,x)))\,ds\,dx\\
      &\qquad +\int_0^t\int_{\TT^d} (\rho\,(\partial_t e_0)(t,x,\rho)+\rho\,u\cdot(\nabla_x e_0)(t,x,\rho))\,dx\,ds,
      \end{split}
\]
where 
\[
\mathcal{E}_0(\rho,u) = |\rho u|^2/2\rho + \rho\,e_0(t,x,\rho),\quad e_0(t,x,\xi)=\int_{\rho_{ref}}^\xi \frac{P_0(t,x,s)}{s^2} \, ds.
\]
Finally if some sequence $P_n$ satisfies uniformly the assumptions $(P1)-(P6)$ then the corresponding solution $\rho_n$ is compact in $L^1([0,\ T]\times\Pi^d)$.
\end{theorem}
\subsection{Existence given $\vartheta$}
We may easily deduce an existence result from~\ref{mainprevious}, by checking that for a given~$\vartheta(t,x)$, the pressure $P(t,x,\rho)=P(\vartheta(t,x),\rho)$, where $P(\vartheta,\rho)$ satisfies~\eqref{P2}--\eqref{P7}, also satisfies (P1)-(P6) above. 
\begin{theorem}
Assume that $\vartheta\in L^{\alpha-\eps'}([0,\ T]\times\Pi^d)\cap L^1([0,\ T],\; W^{\lambda,1}(\Pi^d))$ for some $\lambda>0$ and $\eps'$ small enough. Assume that $P(\vartheta,\rho)$ given by \eqref{virial} with \eqref{P2}--\eqref{P6bis}. Assume moreover that the initial data $m_0$ and $\rho_0\ge 0$ with $\int_{{\mathbb T}^d} \rho_0 = M_0>0$ satisfy
\[
\ce(\rho_0, m_0) = \int_{\TT^{d}} \left(\frac{|m_0|^2}{2\rho_0} + \rho_0 e_P(0,x,\rho_0)\right)\,dx <\infty,\]
where
\[
e_P(t,x,\rho) = \int_{\rho_{ref}}^\rho \frac{P(\vartheta(t,x), s)}{s^2}\,ds.
\]
with $m_0=0$ when $\rho_0=0$.
Then for any $T>0$ there exists a global weak solution to Compressible Navier--Stokes System \eqref{eq:00}--\eqref{eq:01} with the strain tensor given by \eqref{viscous}. More precisely it satisfies the equations in the distribution sense and the bounds
$$ u \in L^2(0,T;H^1(\TT^d)), \qquad \rho\,|u|^2/2 \in L^\infty(0,T;L^1(\TT^d))$$
$$ \rho \in {\mathcal C}([0,T], L^\gamma(\TT^d) \hbox{ weak }) \cap L^p((0,T)\times \TT^d)
    \hbox{ where } 0 < p < \gamma (d+2)/2 -1$$
with the initial conditions in a weak sense and with the heterogeneous pressure state law $P(\vartheta(t,x),.)$ satisfying the energy inequality
\begin{equation}
  \begin{split}
      &\int_{\TT^d} \mathcal{E}_0(\rho,u)\,dx+\int_0^t\int_{\TT^d} \cs:\nabla_x u\,dx\,ds\leq {\mathcal E}(\rho_0,u_0)\\
      &\qquad+\,\int_0^t\int_{\TT^d} \div_x u(s,x)\,(P(\vartheta(s,x),\rho(s,x))-P_0(\rho(s,x)))\,ds\,dx,\\
      \end{split}\label{energy0}
\end{equation}
where 
\[
\begin{split}
&P_0(\rho)=\rho^\gamma+\sum_{n=0}^N \bar B_n\,\rho^n,\\
  &\mathcal{E}_0(\rho,u) = |\rho u|^2/2\rho + \rho\,e_0(\rho),\quad e_0=\int_{\rho_{ref}}^\rho \frac{P_0(s)}{s^2} \, ds.
\end{split}
\]
Finally if some sequence $\vartheta_n$ is uniformly bounded in $L^{\alpha-\eps'}([0,\ T]\times\Pi^d)\cap L^1([0,\ T],\; W^{\lambda,1}(\Pi^d))$, then the corresponding solution $\rho_n$ is compact in $L^1([0,\ T]\times\Pi^d)$.
\label{existrhou}
\end{theorem}
\begin{proof}

\medskip

\noindent {\it Property \rm (P1).}  We note that
\[
|P-P_0|\leq \sum_{n=0}^N |\vartheta(t,x)\,B_n(\vartheta(t,x)) -\bar B_n|\,\rho^n\leq N\,\rho^{\gamma/2}+\sum_{n=0}^N |\vartheta(t,x)\,B_n(\vartheta(t,x)-\bar B_n|^{\gamma/(\gamma-2n)}.
\]
By assumption~\eqref{P6bis}, this implies that
\[
|P-P_0|\leq N\,\rho^{\gamma/2}+C\,\sum_{n=0}^N |\vartheta(t,x)|^{\bar \alpha\,\gamma/(2\,\gamma) -\eps }=N\,\rho^{\gamma/2}+C\,N\,|\vartheta(t,x)|^{\bar \alpha/2-\eps}.
\]
This leads us to define
\[
R(t,x)=\,N\,|\vartheta(t,x)|^{\bar \alpha/2 -\eps},
\]
and we can immediately verify that $R\in L^q_{t,x}$ for some $q>2$ since we assumed that $\vartheta\in L^{\alpha-\eps'}_{t,x}$ with $\bar\alpha\leq \alpha$ for $\eps'$ small enough w.r.t. $\eps$.

\medskip

\noindent {\it Property \rm (P2).}  We can check (P2) almost immediately as well by taking $\vartheta_1(t,x)=C$ for some large constant $C$, as for example
\[
|P_0-\rho^\gamma|\leq \sum_{n=0}^N |\bar B_n|\,\rho^n\leq C+C\,\rho^N,
\]
where we recall that $N\leq \gamma/2$.

\medskip

\noindent {\it Properties \rm (P3) \it and \rm (P4).}
As indicated in the statement of the theorem, we take
$P_0 = \phi'(\rho)\rho - \phi(\rho)$ where $\phi$ is given by \eqref{phi}, that is
\[
P_0(t,x,\rho)=\rho^\gamma+\sum_{n=0}^N \bar B_n\,\rho^n.
\]
This directly implies (P3) and (P4) since $P_0$ does not explicitly depends on $t$ or $x$ and thus $\partial_t P_0=0$ and $\nabla_x P_0=0$. Consequently we also have that $\partial_t e_0=0$ and $\nabla_x e_0=0$ so that we do not have the corresponding terms in the energy equality.

\medskip
\noindent {\it Properties \rm (P5) \it and \rm (P6).}  Observe that
\[
\begin{split}
  &|P(t,x,z)-P(t,y,w)|\leq C\,(z^{\gamma-1}+w^{\gamma-1})\,|z-w|\\
  &\qquad+C\,\sum_{n=0}^N |\vartheta(t,x)\,B_n(\vartheta(t,x))+\vartheta(t,y)\,B_n(\vartheta(t,y))|\,(z^{n-1}+w^{n-1})\,|z-w|\\
  &\qquad+C\,\sum_{n=0}^N |\vartheta(t,x)\,B_n(\vartheta(t,x))-\vartheta(t,y)\,B_n(\vartheta(t,y))|\,(z^{n}+w^{n}).
\end{split}
\]
Therefore
\begin{equation}
\begin{split}
  &|P(t,x,z)-P(t,y,w)|\\
  &\leq C\,\left(z^{\gamma-1}+w^{\gamma-1}+\sum_{n=0}^N |\vartheta(t,x)\,B_n(\vartheta(t,x))+\vartheta(t,y)\,B_n(\vartheta(t,y))|^{\frac{\gamma-1}{\gamma-n}}\right)\,|z-w|\\
  &\qquad+C\,\sum_{n=0}^N |\vartheta(t,x)\,B_n(\vartheta(t,x))-\vartheta(t,y)\,B_n(\vartheta(t,y))|\,(z^{n}+w^{n}).
\end{split}\label{diffP}
\end{equation}
We can hence choose any $\tilde P$ s.t. 
\[
\tilde P(t,x)\geq\sum_{n=0}^N |\vartheta(t,x)\,B_n(\vartheta(t,x))-\bar B_n|^{\frac{\gamma-1}{\gamma-n}},
\]
or from~\eqref{P6bis} again, for some small $\eps>0$,
\[
\tilde P(t,x)\geq\sum_{n=0}^N |\vartheta(t,x)|^{\bar\alpha\frac{(\gamma-2n)(\gamma-1)}{(\gamma-n)\,2\,\gamma}-\eps}.
\]
Of course, since 
\[
\frac{(\gamma-2n)(\gamma-1)}{(\gamma-n)\,2\,\gamma}\leq 1,\quad\mbox{as}\ \gamma^2-2n\,\gamma-\gamma+2n\leq 2\,\gamma^2-2n\gamma,
\]
we may simply take
\[
\tilde P(t,x)=C\,|\vartheta(t,x)|^{\bar\alpha-\eps}.
\]
Since $\vartheta\in L^{\alpha-\eps'}_{t,x}$ and $\alpha\geq \bar\alpha$, this immediately imply that $\tilde P\in L^{s_0}_{t,x}$ for some $s_0>1$, provided again that $\eps'$ is small enough. Moreover since $\vartheta\in L^1_t W^{\lambda,1}_x$ for some $\lambda>0$, by interpolation we deduce that $\tilde P\in L^{s_0}_t W^{\tilde\lambda,s_0}_x$ for some $\tilde\lambda>0$ and for some $s_0>1$. This directly implies property (P6) on $\tilde P$.

From~\eqref{diffP}, we take
\[
Q(t,x,y,z,w)=C\,\sum_{n=0}^N |\vartheta(t,x)\,B_n(\vartheta(t,x))-\vartheta(t,y)\,B_n(\vartheta(t,y))|\,(z^{n}+w^{n}).
\]
Consider now any sequence $\rho_k$ uniformly bounded in $L^\infty_t L^\gamma_x$. We may directly bound for $s_1>1$ small enough
\[
\|Q(t,x,y,\rho_k(t,x),\rho_k(t,y)\|_{L^{s_1}_{t,x}}\leq C\,\|\rho_k\|_{L^\infty_t L^\gamma_x}\,\sum_{n=0}^N \|\vartheta(t,x)\,B_n(\vartheta(t,x))-\bar B_n\|_{L^{s_1}_t L^{\gamma/(\gamma-ns_1)}_x}.
\]
Still using assumption~\eqref{P6bis}, we have that
\[
\|\vartheta(t,x)\,B_n(\vartheta(t,x))-\bar B_n\|_{L^{s_1}_t L^{\gamma/(\gamma-n\,s_1)}_x}\leq \|\vartheta\|_{L^{s_1\,\left(\bar\alpha\,\frac{\gamma-2n}{2\gamma}-\eps\right)}_t
	 L^{s_1\left(\,\bar\alpha\,\frac{\gamma-2n}{2(\gamma-ns_1)}-\eps\right)}_x}^{\bar\alpha\,\frac{\gamma-2n}{2\gamma}-\eps}.
\]
Clearly both $\frac{\gamma-2n}{2\gamma}\leq1/2$ and $\frac{\gamma-2n}{2(\gamma-ns_1)}\leq 1/2$ as long as $s_1<2$. Since $\vartheta\in L^{\alpha-\eps'}_{t,x}$, we can take up to $s_1=2$ and have $Q(t,x,y,\rho_k(t,x),\rho_k(t,y))\in L^{s_1}_{t,x}$ uniformly in $k$.

We may similarly prove property (P6) for $Q$,
\[\begin{split}
&\int_0^T \int_{\Pi^{2d}} \frac{K_h(x-y)}{\|K_h\|_{L^1}}\, |Q(t,x,y,\rho_k(t,x),\rho_k(t,y))|^{s_1}\leq C\,\|\rho_k\|_{L^\infty_t L^\gamma_x}\\
&\qquad\qquad \times\sum_{n=0}^N \int_0^T \left(\int_{\Pi^{2d}} \frac{K_h(x-y)}{\|K_h\|_{L^1}}\,|\vartheta(t,x)\,B_n(\vartheta(t,x))-\vartheta(t,y)\,B_n(\vartheta(t,y)|^{\frac{\gamma}{\gamma-ns_1}}\right)^{s_1\,\frac{\gamma-ns_1}{\gamma}}.
\end{split}
\]
From our previous argument we know that $\vartheta\,B_n(\vartheta)-\bar B_n$ does belong to $L^{s_1}_t L^{\frac{\gamma}{\gamma-ns_1}}_x$ and in fact to some $L^p_t L^q_x$ with $p>s_1$ and $q>\frac{\gamma}{\gamma-ns_1}$. Moreover since $\vartheta\in L^1_t W^{\lambda,1}$ and $B_n$ is locally Lipschitz from~\eqref{P6bis}, we also have that $\vartheta\,B_n(\vartheta)-\bar B_n\in L^1_t W^{\lambda',1}$ for some $\lambda'>0$. By interpolation, this finally implies that $\vartheta\,B_n(\vartheta)-\bar B_n\in L^{s_1}_t W^{\lambda'',\frac{\gamma}{\gamma-ns_1}}_x$ which proves (P6) for $Q$.

\end{proof}
\section{Second step: Solve the temperature equation with $\rho$, $u$ given}
We first start to solve an equation related to the energy and then use an implicit function 
procedure to find the corresponding temperature. This is the important and new part in the
global existence construction procedure.
%
\subsection{An equivalent system with good unknowns}
 
\noindent From~\eqref{max:equ}, it is straightforward to check that 
\begin{equation}
	e = m(\vartheta) + \frac{\rho^{\gamma-1}}{\gamma-1} -\vartheta^2\sum_{N\geq n\geq2}B_n'(\vartheta)\frac{\rho^{n-1}}{n-1} + \vartheta^2B_0'(\vartheta)\frac{1}{\rho}.
\end{equation}
Instead of working on the system involving the quantity $\rho e$, we present here an equivalent system with what will prove to be an easier unknown to handle
\begin{equation}
	g = \rho \ee, 
	\label{g}
\end{equation}
where $\ee$ is given by
\begin{equation}
	\label{new:e}
	\ee = -\vartheta^2\sum_{2\leq n\leq N}B_n'(\vartheta)\frac{\rho^{n-1}}{n-1} + \vartheta^2B_0'(\vartheta)\frac{1}{\rho},
\end{equation}
where we recall that $N<\gamma/2$. Define a new pressure $\pp$ by
\begin{equation}
	\label{new:p}
	\pp = \vartheta\,\sum_{n=0}^{N} B_n(\vartheta)\,\rho^n.
\end{equation}
Then the good unknown $g$ satisfies 
\begin{equation}
	\label{equ:g}
	\partial_{t}g + \div(g u) + \pp\div u = \cs:\nabla u + \div (\kappa(\vartheta)\nabla\vartheta),
\end{equation}
where $\cs:\nabla u = \Tr(\cs\nabla u)$ as before.
\begin{remark}
	From the assumption~\ref{P4} on $P$, it follows easily that $\vartheta^2B_n'\le 0$ for $n\ge2$, which implies that $B_n$ is a decreasing function in $\vartheta$ for $n\ge2$. Moreover, in view of assumptions \ref{P2} and \ref{P4}, we have that
	\begin{equation*}
	\ee>0,\ \ \ \ \  \frac{\partial e}{\partial \vartheta}=\frac{\partial \ee}{\partial \vartheta} > 0, \ \ \ \ \ \ \mbox{for}\ \vartheta>0.
	\end{equation*}
\end{remark}
\subsection{The solvability of the quasi-linear parabolic system}
Consider the equation
\begin{equation}
	\label{qus:par}
	\partial_t f - \sum_i\frac{\partial}{\partial x_i}(a_i(t, x, f, \nabla f)) + \ao(t, x, f, \nabla f) = 0   \ \ \ \ \ \ (t, x)\in Q_T = [0, T]\times\TT^d 
\end{equation}
with the initial condition  
\begin{equation}\label{bcf}
f\vert_{t=0} = f_0.
\end{equation} 
We recall here the classical assumptions on the functions $a_0$ and $a=(a_1, a_2, \dots, a_d)(t,x,f,p)$ for $t\in [0,\ T]$, $x\in \TT^d$, $f\in \RR$ and $p\in\RR^d$ to obtain a solution $f(t,x)$. 
\begin{theorem}[\cite{LaSoUr}]
	\label{exi:div}
Suppose that
\begin{itemize}
	\item {\rm (H1).} The system~\eqref{qus:par} is parabolic in the sense that
	\begin{equation}
		\label{ass:par}
		c_1(f)|\xi|^2 \le \sum_{i,j}\frac{\partial a_i}{\partial p_j}\xi_i\xi_j \le c_2(f)|\xi|^2,
	\end{equation}
where $c_0$, $c_1$, and $c_2$ are positive, continuous and potentially depend on $f$.
	\item {\rm (H2).}  For $(t, x)\in \bar{Q_T}$ and for any $f$ and $p$, the inequality
	\begin{equation}
		\label{bd1:a}
		\sum_{i} |a_i(t, x, f, p)| + |\ao(t, x, f, p)| \le b(|f|, |p|) \phi_1(t, x)
	\end{equation}
holds with a continuous function $b$ and a function $\phi_1\in L^1(Q_T)$. 
    \item {\rm (H3).} With $|f|\le M$ where $M>0$ is a constant large enough 
    and arbitrary $p$ , we have the bound 
    \begin{equation}
    	\label{bd2:a}
    	\sum_{i=1}^{n}\left(|a_i|+\left|\frac{\partial a_i}{\partial f}\right|\right)(1+|p|) 
    	+ \sum_{i, j=1}^{n}\left|\frac{\partial a_i}{\partial x_j}\right| + |\ao| \le c_3(1+|p|^2).
    \end{equation}
for some $c_3>0$.
\item {\rm (H4).}  The functions $a_i$, $\partial a_i/\partial p_j$, $\partial a_i/\partial x_j$, and $\partial a_i/\partial f$ are H\"older continuous with exponent $\beta$, $\beta/2$, $\beta$, and $\beta$ respectively.
\item {\rm (H5).}  The following bounds holds,
\begin{equation}
	\label{bd3:a}
	\left|   \frac{\partial a_i}{\partial f},\ \frac{\partial a_i}{\partial t}, \frac{\partial \ao}{\partial p}, \ \frac{\partial \ao}{\partial f}, \ \frac{\partial \ao}{\partial t}   \right|
	\le \phi_2(t,x)
\end{equation}
for any $|f|, |p|\le M$ for some sufficiently large constant $M$, where $\phi_2(t,x) \in L^{r, p}$ with $r, p\ge2$.
\end{itemize}
Assuming $f_0 \in C_x^{2+\beta}$, then there exists a unique solution $f$ of the system~\ref{qus:par} such that $f\in C_t^{1+\beta/2}C_x^{2+\beta}$. Moreover, we have $\partial_{t,x} f\in L^2$.
\end{theorem}
\subsection{Solving an approximate system}
In order to solve the system, one way is to see \eqref{equ:g} as a quasi-linear parabolic equation of the unknown function $g$, namely $\vartheta=\vartheta(\rho, g)$. 
Equation~\eqref{equ:g} is in the right form since
\begin{equation}
	\label{div:form:g}
	\partial_t g - \sum_i \frac{\partial}{\partial x_i}(a_i(t, x, g, \nabla g)) + \ao(t, x, g, \nabla g) = 0,
\end{equation}
where
\begin{equation}
	\label{term:ai}
	a_i = -gu_i + (\kappa(\vartheta)\nabla\vartheta)_i
	     = -gu_i + \kappa(\vartheta)\frac{\partial \vartheta}{\partial g}\partial_i g + \kappa(\vartheta)\frac{\partial \vartheta}{\partial \rho}\partial_i \rho
\end{equation}
and
\begin{equation}
	\label{term:a}
     \ao = \pp\div u -\cs:\nabla u.
\end{equation}
There are however several regularity issues when trying to apply directly Theorem~\ref{exi:div}, which forces us to introduce several approximations.
First of all, there is a singularity in $\partial \vartheta/\partial g$ when $g$ or $\vartheta$ is close to $0$. Second the assumptions on the various functions $a_i$ and $a_0$ require some additional regularity on $\rho$ and $u$.

This leads us to look at an  approximate system where we modify the relation between $g$ and $\vartheta$. More specifically, for a given $\eps>0$, we first solve in $g_\eps$, the system 
\begin{equation}
	\label{app}
	\partial_tg_\eps - \sum_i\frac{\partial}{\partial x_i}(a_i^\eps(t, x, g_\epsilon, \nabla g_\epsilon)) + a^\eps_0(t, x, g_\epsilon, \nabla g_\epsilon) = 0,
\end{equation}
with
\begin{equation}
	a_i^\eps (t, x, g_\epsilon, \nabla g_\epsilon) 
	= -g_\eps u_i + \kappa(\vartheta_\eps)\frac{\partial \vartheta_\eps}{\partial g_\eps}\partial_i g_\eps + \kappa(\vartheta_\eps)\frac{\partial \vartheta_\eps}{\partial \rho}\partial_i \rho,\label{aieps}
\end{equation}
and 
\begin{equation}
	a^\eps_0(t, x, g_\epsilon, \nabla g_\epsilon)= \pp_\eps\div u -\cs:\nabla u.\label{a0eps}
\end{equation}
However we take
\begin{equation}
g_\epsilon = \rho \ee_\eps,\qquad 	\ee_\eps =\eps\,\frac{\vartheta_\eps}{\rho}  -\vartheta_\eps^2\sum_{N\geq n\geq2}B_n'(\vartheta_\eps)\frac{\rho^{n-1}}{n-1} + \vartheta_\eps^2\,B_0'(\vartheta_\eps)\frac{1}{\rho},\label{gtothetaeps}
\end{equation}
which changes the relation between $\vartheta_\eps$ and $g_\eps$, resolves the degeneracy around $\vartheta=0$ and also implicitly modifies the $a_i$.

Finally, we  adapt
 the pressure $\pp_\eps$ to match the new energy and keep
\begin{equation}
	\pp_\eps = -\eps\,\vartheta_\eps\,\log\vartheta_\eps+ \vartheta_\eps\,\sum_{n=0}^{N} B_n(\vartheta_\eps)\,\rho^n.\label{Peps}
\end{equation}
We then have the following existence theorem for the approximate equation.
\begin{theorem}
	\label{exi:g}
	Let $\pp$ be defined in~\eqref{Peps} and $\vartheta_\eps$ be defined in term of $g_\eps$ through \eqref{gtothetaeps} for some $\rho\in C^1_t C^{2+\beta}_x$ and $u\in C^1_t C^{2+\beta}_x$. Then for any initial data $g_0>0 $ with $g_0\in C_x^{2+\beta}$, there exists a unique classical solution $g\in C_t^{1+\beta/2}C_x^{2+\beta}$ to the system~\eqref{div:form:g} where the $a_i^\eps$ and $a_0^\eps$ are given respectively by \eqref{aieps} and \eqref{a0eps}. 
\end{theorem}
\begin{proof}
To simplify the notations, within this proof, we omit the $\eps$ subscript as it will not cause any confusion; we take the limit $\eps\to 0$ in the next subsection. We use Theorem~\ref{exi:div} to prove the existence result. 

\noindent {\it Important relations between $g$ and $\vartheta$.}
Through \eqref{gtothetaeps}, we first observe that $\vartheta$ can be seen as $\vartheta(t,x,g)$ or $\vartheta(\rho,g)$. This can be proved by showing that $g$ is strictly increasing in $\vartheta$ by differentiating~\eqref{gtothetaeps}.
We find that  
	\begin{equation}
		\label{der:g:the}
		\frac{\partial g}{\partial \vartheta} =\eps-\sum_{2\le n\le N}\frac{d}{d \vartheta}(\vartheta\,^2B_n'(\vartheta))\frac{\rho^n}{n-1} + \frac{d}{d \vartheta}(\vartheta^2B_0'(\vartheta)),
	\end{equation}
and	by assumptions~\ref{P2} and \ref{P4}
	we easily get a lower bound for $\partial g/\partial \vartheta $
	\begin{equation*}
			\frac{\partial g}{\partial \vartheta} \ge \eps+ 2\,\vartheta B_0' + \bar\vartheta^2 B_0'' \ge	\eps+\frac{1}{C}\vartheta^{\gamma_{\vartheta}-1}.
	\end{equation*}
 	Using assumptions~\ref{P2}, \ref{P4}, and \ref{P5}, we may further deduce an upper bound for $\partial g/\partial \vartheta $ as
 	\begin{equation*}
 	 \frac{\partial g}{\partial \vartheta} \le \eps+C\, \sum_{0\le n\le N} \vartheta^{(\gamma-n)\gamma_{\vartheta}/\gamma-1}\,\|\rho\|_{L^\infty}^n
 		\le \eps +C\,(\vartheta^{\gamma_{\vartheta}-1} + \vartheta^{(\gamma-N)\gamma_{\vartheta}/\gamma-1}).
 	\end{equation*}
 Combining the above two inequalities gives
	\begin{equation}
		\label{bd:g:the}
	\eps+\frac{1}{C}	\vartheta^{\gamma_{\vartheta}-1} \le	\frac{\partial g}{\partial \vartheta} 
		\le \eps+C\,(\vartheta^{\gamma_{\vartheta}-1} + \vartheta^{(\gamma-N)\gamma_{\vartheta}/\gamma-1}).
	\end{equation}
From the definition of $g$ in~\eqref{gtothetaeps}, we have that $g=0$ if $\vartheta=0$.  As a consequence,
\begin{equation}
	\label{bd:g}
	\eps\,\vartheta+\frac{1}{C}\,\vartheta^{\gamma_{\vartheta}} \le	g
	\le \eps\,\vartheta+C\,(\vartheta^{\gamma_{\vartheta}} + \vartheta^{(\gamma-N)\gamma_{\vartheta}/\gamma}).
\end{equation}
We also need an upper bound of $\partial \vartheta/\partial g$ in term of $g$ as
\begin{equation}
	\label{bd:the:g}
	 	\frac{\partial \vartheta} {\partial g} = \left(\frac{\partial g}{\partial \vartheta} \right)^{-1}
	 	\le
	 	\frac{C}{\eps+\vartheta^{\gamma_{\vartheta}-1}} \le \frac{C}{\eps},
\end{equation}
together with a lower bound
\begin{equation}
	\label{lbd:the:g}
	 	\frac{\partial \vartheta} {\partial g} = \left(\frac{\partial g}{\partial \vartheta} \right)^{-1}
	 	\geq
	 	\frac{1}{\eps+C\,\vartheta^{\gamma_{\vartheta}-1}+C\,\vartheta^{(\gamma-N)\,\gamma_0/\gamma}} \geq
\frac{1}{M(g)}
\end{equation}
with $M\ge \eps$ being a smooth function of $g$, where
we used~\eqref{bd:g:the} and \eqref{bd:g}.

\medskip

\noindent {\it Hypothesis \rm (H1).}
We also note, for further use, that~\eqref{con:ass} yields
\begin{equation}
	\label{bd:con}
1\leq	\kappa(\vartheta) \le
	C\,(g^{\alpha/\gamma_{\vartheta}}+1).
\end{equation}
We can then make explicit the various coefficients with for example
\[
a_i(t,x,g,p)=a_i(\rho,\nabla\rho,u,g,p)=-g\,u_i+\kappa(\vartheta(\rho,g))\, \frac{\partial\vartheta}{\partial g}\,p_i+\kappa(\vartheta(\rho,g))\, \frac{\partial\vartheta}{\partial g}\,\partial_i \rho.
\]
This directly implies that
\begin{equation*}
	\frac{\partial a}{\partial p}\bigg|_{t,x} = \kappa(\vartheta)\frac{\partial \vartheta}{\partial g}\bigg|_{t,x} I_d =  \kappa(\vartheta)\left(\frac{\partial g}{\partial \vartheta}\right)^{-1}\bigg |_{t,x} I_d,
\end{equation*}
where $I_d$ is the identity matrix. Therefore, we obtain
\begin{equation*}
	c_1 |\xi|^2 \le \xi^T\frac{\partial a}{\partial p}\bigg|_{\rho} \xi \le c_2 |\xi|^2
\end{equation*} 
for any $\xi\in\RR^d$ provided $|g|\ge c_0$ where $c_1$ and $c_2$ are non-vanishing and depend polynomially  on $g$, which verifies \eqref{ass:par}. 

\medskip

\noindent {\it Hypothesis \rm (H2).} From the equation
\begin{equation*}
	g =\eps\,\vartheta  -\sum_{2\le n\le N}\vartheta^2B_n'(\vartheta)\frac{\rho^n}{n-1} + \vartheta^2B_0'(\vartheta),
\end{equation*}
again viewing $\vartheta$ as a function of both $g$ and $\rho$, we take derivative with respect to $\rho$, keeping $g$ fixed, to get
\begin{equation*}
	0 =\eps\,\frac{\partial\vartheta}{\partial\rho}  -\sum_{2\le n\le N}\frac{d}{d \vartheta}(\vartheta^2B_n'(\vartheta))\frac{\partial \vartheta}{\partial \rho}\frac{\rho^n}{n-1} + \frac{d}{d \vartheta}(\vartheta^2B_0'(\vartheta))\frac{\partial \vartheta}{\partial \rho} -\sum_{2\le n\le N}\vartheta^2B_n'(\vartheta)\frac{n}{n-1}\rho^{n-1}.
\end{equation*}
from where one obtains
\begin{equation*}
	\frac{\partial\vartheta}{\partial\rho} = \frac{\sum_{2\le n\le N}\frac{n}{n-1}\vartheta^2B_n'(\vartheta)\rho^{n-1}}{\eps-\sum_{2\le n\le N}\frac{d}{d \vartheta}(\vartheta^2B_n'(\vartheta))\frac{\rho^n}{n-1} + \frac{d}{d \vartheta}(\vartheta^2B_0'(\vartheta))}
	= 
	\sum_{2\le n\le N}\frac{n}{n-1}\vartheta^2B_n'(\vartheta)\rho^{n-1} \left( \frac{\partial g}{\partial \vartheta}\right)^{-1}.
\end{equation*}
By \eqref{bd:the:g}, assumptions~\ref{P4} and \ref{P5}, it holds 
\begin{align}
	\label{bd:the:rho}
	\left| \frac{\partial\vartheta}{\partial\rho} \right| 
	&\le
	\frac{C}{\eps+\vartheta^{\gamma_{\vartheta}-1}}\,\left| \sum_{2\le n\le N}\frac{n}{n-1}\vartheta^2B_n'(\vartheta)\rho^{n-1}\right|
	\nn&\le
	\frac{C}{\eps+\vartheta^{\gamma_{\vartheta}-1}}\,(\vartheta^{\gamma_{\vartheta}} + \vartheta^{(\gamma-N)\gamma_{\vartheta}/\gamma})
	\nn&
	\le \frac{C}{\eps}+C\,\vartheta \le \frac{C}{\eps}+C\,g^{1/\gamma_{\vartheta}}.
\end{align}
Combining~\eqref{term:ai} with the estimates~\eqref{bd:the:g}, \eqref{bd:con}, and \eqref{bd:the:rho}, we further get
\begin{align}
	\label{bd:ai}
	\sum_{i} |a_i(t,x,g,p)| \le& C\,g\,\|u\|_{L^\infty} + \,\frac{Cg^{\alpha/\gamma_{\vartheta}}+C}{\eps}\,|p|
	\nn&
	+ C(g^{\alpha/\gamma_{\vartheta}}+1)\,(\frac{1}{\eps}+g^{1/\gamma_{\vartheta}})\, \|\nabla \rho\|_{L^\infty}.
\end{align}
%
In view of~\ref{a0eps}, \eqref{Peps}, assumptions~\ref{P2}, and \ref{P5}, we similarly obtain an upper bound of $\ao$ as
\begin{equation}
	\label{bd3:a0}
	|\ao| \le C\,(1+\vartheta^{\gamma_{\vartheta}} + \vartheta^{(\gamma-N)\gamma_{\vartheta}/\gamma})\,\|\div u\|_{L^\infty} + \|S:\nabla u\|_{L^\infty}
	\le C g + C,
\end{equation}
by using again the regularity of $\rho$ and $u$. Hence the condition~\eqref{bd1:a} is verified by collecting the estimates~\eqref{bd:ai} and \eqref{bd3:a0}.

\medskip

\noindent {\it Hypothesis \rm (H3)--(H4).}
Next we turn to the verification of~\eqref{bd2:a}. First, we compute the derivative of $a_i$ as
\begin{equation*}
	\frac{\partial a_i(t,x,g,p)}{\partial g} = -u_i +  \kappa'(\vartheta)\left(\frac{\partial \vartheta}{\partial g}\right)^2\,p
	+ \kappa(\vartheta)\frac{\partial^2 \vartheta}{\partial g^2}\,p 
	+  \kappa'(\vartheta)\frac{\partial \vartheta}{\partial g}\frac{\partial \vartheta}{\partial \rho}\nabla\rho
	+ \kappa(\vartheta)\frac{\partial^2 \vartheta}{\partial g \partial \rho}\nabla\rho
\end{equation*}
From~\eqref{der:g:the}, it is straightforward to get
\begin{equation*}
	\frac{\partial \vartheta}{\partial g} = \frac{1}{\eps -\sum_{2\le n\le N}\frac{d}{d\vartheta}\,(\vartheta^2B_n'(\vartheta))\,\rho^n/(n-1) + \frac{d}{d\vartheta}(\vartheta^2B_0'(\vartheta))},
\end{equation*}
which leads to
\begin{equation*}
	\frac{\partial^2 \vartheta}{\partial g^2} = \frac{-\sum_{2\le n\le N}\frac{d^2}{d\vartheta^2}(\vartheta^2B_n'(\vartheta))\,\rho^n/(n-1) + \frac{d^2}{d\vartheta^2}(\vartheta^2B_0'(\vartheta))}{ \left(\eps -\sum_{2\le n\le N}\frac{d}{d\vartheta}\,(\vartheta^2B_n'(\vartheta))\,\rho^n/(n-1) + \frac{d}{d\vartheta}(\vartheta^2B_0'(\vartheta))\right)^2}\; \frac{\partial \vartheta}{\partial g}.
\end{equation*}
By combining our previous bounds, we can prove that 
\begin{equation*}
		\sum_{i=1}^{n}\left(|a_{i}|+\left|\frac{\partial a_{i}}{\partial g }\right|\right)(1+|p|) 
	 + |a | \le c_3(g)\,(1+|p|^2)
\end{equation*}
for some $c_3>0$ which depends on $\eps$ and is polynomial in $g$ and hence bounded whenever $g\leq M$.  We may perform again similar calculations for all $\partial a_{i }/\partial x_j$ which yields the bound~\eqref{bd2:a}. The same formula and the regularity of $\rho$ and $u$ ensures that $a_i$, $\partial a_i/\partial p_j$, $\partial a_i/\partial x_j$, and $\partial a_i/\partial f$ are H\"older continuous with respect to $t, x, g,$ and $p$. We can check the bound~\eqref{bd3:a} in the same manner.

This satisfies all assumptions of Theorem~\ref{exi:div} as long as we can ensure that $g> 0$. This follows from a straightforward maximum principle applied to any classical solution 
of~\eqref{div:form:g}: See the positivity part just below.

\medskip

\noindent {\it Positivity of $g$.} We note that we can rewrite \eqref{div:form:g} as
\[
\partial_t g+\div (g\,U)=\nabla_x\left(\kappa(\vartheta)\,\frac{\partial \vartheta}{\partial g}\right)\cdot \nabla g+\kappa(\vartheta)\,\frac{\partial \vartheta}{\partial g}\,\Delta g+S:\nabla u-\tilde P\,\div u,
\]
where
\[
U=u+\kappa(\vartheta)\,\nabla\rho\,\frac{\partial_\rho \vartheta}{g}.
\]
Remark that $S:\nabla u\geq 0$. Moreover $\vartheta=0$ when $g=0$ and for $0\leq g\leq 1$, we have that
\[
|\tilde P|\leq C\,\vartheta\leq \frac{C}{\eps}\,g,
\]
while by~\eqref{bd:the:rho}, we have still for $0\leq g\leq 1$ that
\[
\left|\frac{\partial\vartheta}{\partial\rho}\right|\leq C\,\vartheta^2\leq \frac{C}{\eps}\,g^2.
\]
This ensures that the solution $g$ and then $\vartheta$ are both strictly positive.
\end{proof}
\subsection{Existence of solutions $\vartheta$ such that the entropy defined by \eqref{entrop} satisfies the equation  \eqref{ENTROPY}.}
From the previous result, we may pass to the limit $\eps\to 0$ to obtain the following existence result.
\begin{theorem}
Assume that $\rho^0\in L^\gamma_x$, $\vartheta^0\in L^{\gamma_\vartheta}_x$. Assume moreover that $\rho\in L^\infty([0,\ T],\;L^\gamma(\Pi^d))$ and $u\in L^2([0,\ T],\;H^1(\Pi^d))$ and solve the continuity equation~\eqref{eq:00}. Then there exists $\vartheta\in L^\infty(0,T,\;L^{\gamma_\vartheta}(\Pi^d))\cap L^\alpha([0,\ T]\times\Pi^d)$ such that 
\begin{equation}
\sup_t \int_{\TT^d}  \vartheta^{\gamma_\vartheta}\,dx +\int_0^T \int_{\TT^d}  \kappa(\vartheta)\,\frac{|\nabla \vartheta|^2}{\vartheta^2}\,dx\,dt+\int_0^T \int_{\TT^d}  \frac{|\nabla u|^2}{\vartheta}\,dx\,dt\leq C(\|\rho\|_{L^\infty_t L^\gamma_x},\|u\|_{L^2_tH^1_x}),\label{aprioritheta}
\end{equation}
for some constant depending on the norms $(\|\rho\|_{L^\infty_t L^\gamma_x},\,\|u\|_{L^2_tH^1_x}$ and on the initial data $\rho^0$ and $\vartheta^0$). 
Furthermore, defining the entropy through the relation
\begin{equation}
\rho\,s=\rho\,\left(-\sum_{n=2}^N \tilde B_n(\vartheta)\,\frac{\rho^{n-1}}{n-1}+\frac{1}{\rho}\,\tilde B_0(\vartheta)\right), \quad \tilde B_n'(\vartheta)=\vartheta\,B_n''(\vartheta)+2\,B_n'(\vartheta),
\label{defentropy}
\end{equation}
then
\begin{equation}
  \|\rho\,s\|_{L^\infty_t L^1_x}+\|\rho s\,u\|_{L^1_{t,x}}\leq C(\|\rho\|_{L^\infty_t L^\gamma_x},\|u\|_{L^2_tH^1_x}),\label{entropybounded}
  \end{equation}
and $\rho\,s$ solves the following inequation in the sense of distributions
\begin{equation}
  \begin{split}
  &\partial_t(\rho\,s)+\div(\rho\,s\,u)\geq \nabla_x\left(\frac{\kappa(\vartheta)}{\vartheta}\,\nabla_x\vartheta\right) +\frac{1}{\vartheta}\,\cs:\nabla u+\kappa(\vartheta)\,\frac{|\nabla \vartheta|^2}{\vartheta^2}
\end{split}\label{eqslim}
  \end{equation}
with an initial condition satisfied weakly through the following inequality
\begin{equation}\label{inis}
 \rho s \vert_{t=0+} \ge \rho_0 s(\rho_0,\vartheta_0).
\end{equation}
Finally defining $g$ through the identities~\eqref{g}-\eqref{new:e}, we also have  the energy equality
\begin{equation}
\int_{\TT^d} g(t,x)\,dx=\int_{\TT^d} g^0(x)\,dx+\int_0^t \int_{\TT^d} (\cs:\nabla u-\tilde P\,\div u)(s,x)\,ds\,dx.\label{energyequality}
  \end{equation}\label{solvevartheta}
\end{theorem}

\begin{proof}
  The strategy of the proof is straightforward. Given $u\in L^2_t H^1_x$, we construct $u_\eps\in C^\infty$ that converges to $u$ strongly in $L^2_t H^1_x$. For simplicity, we consider here an approximation by convolution. Given $\rho\in L^\infty_t L^\gamma_x$, we construct $\rho_\eps\in C^\infty$ by convoluting through the same kernel, uniformly bounded in $L^\infty_t L^\gamma_x$ and converging to $\rho$ in $L^p_t L^\gamma_x$ for every $p<\infty$. Observe that the standard commutator estimate implies that
  \[
  \partial_t\rho_\eps+\div(\rho_\eps\,u_\eps)=R_\eps,
  \]
  where $R_\eps\to 0$ in $L^2_t L^p_x$ with $1/p=1/2+1/\gamma$. 
 
We also choose $\vartheta_\eps^0=\eps+\vartheta^0$ which is uniformly in $L^{\gamma_\vartheta}_x$.  For any fixed $\eps>0$, we then obtain a classical solution $g_\eps$ to~\eqref{div:form:g} with \eqref{aieps}-\eqref{a0eps}. We then have to pass to the limit as $\eps\to 0$ in the system.

\medskip

\noindent {\it Uniform bounds.}
  The critical point is hence to derive appropriate estimates on $g_\eps$ and $\vartheta_\eps$ that are uniform in $\eps$. This is naturally based on  equivalent energy and entropy estimates.
  To start with the energy, by directly integrating the equation on $g_\eps$ first in space and then in time, we obtain 
\[
\int_{\TT^d} g_\eps(t,x)\,dx=\int_{\TT^d} g_\eps^0(x)\,dx+\int_0^t\int_{\TT^d} (\cs_\eps:\nabla u_\eps-\tilde P_\eps\,\div u_\eps)(s,x)\,dx\,ds.
\]
From our assumptions on the initial conditions $\rho^0$ and $\vartheta^0$, we  have uniform bounds on $g_\eps^0$ in $L^1$. Indeed for any $n$, $(\rho_\eps^0)^{n}$ converges to $(\rho^0)^n$ strongly in $L^{\gamma/n}_x$. From assumption~\eqref{P6}, $(\vartheta_\eps^0)^2\,B_n'(\vartheta_\eps^0)$ converges strongly in $L^{\gamma/(\gamma-n)}$. Since $\frac{\gamma-n}{\gamma}+\frac{n}{\gamma}=1$, we directly obtains that $g_\eps^0$ converges to $g^0$ strongly in $L^1_x$ and that it is uniformly bounded in $L^1_x$.

Moreover by convexity of the $H^1$ norm, we also have that
\[
\int_0^t\int_{\TT^d} \cs_\eps:\nabla u_\eps\,dx\,ds\leq \|u\|_{L^2_t H^1_x}^2.
\]
This yields that
\begin{equation}
\begin{split}
\int_{\TT^d} 
g_\eps(t,x)\,dx \leq &C+\|u\|_{L^2_t H^1_x}^2+\|u\|_{L^2_t H^1_x}\,\|\tilde P_\eps\|_{L^2_{t,x}}.\\
\end{split}\label{apriorigeps}
  \end{equation}
It remains to control the norm in the right-hand side. From the definition of $\tilde P_\eps$, we have that
\[
\|\tilde P_\eps\|_{L^2_{t,x}}\leq \sum_{n=0}^N \|\vartheta_\eps\,B_n(\vartheta_\eps)\,\rho_\eps^n\|_{L^2_{t,x}}+\eps\,\|\vartheta_\eps\,\log \vartheta_\eps\|_{L^2_{t,x}}.
\]
From the $L^\infty_t L^\gamma_x$ bound on $\rho$, this implies that
\[
\|\tilde P_\eps\|_{L^2_{t,x}}\leq \sum_{n=0}^N \|\vartheta_\eps\,B_n(\vartheta_\eps)\|_{L^2_t L^{q_n}_x}\,\|\rho_\eps\|_{L^\infty_{t} L^\gamma_x}^{n}+\eps\,\|\vartheta_\eps\,\log \vartheta_\eps\|_{L^2_{t,x}},
\]
with $1/q_n=1/2-n/\gamma$ or $q_n=2\,\gamma/(\gamma-2n)$. 

We may now use assumption~\eqref{P6bis} to further bound
\[
\|\vartheta_\eps\,B_n(\vartheta_\eps)\|_{L^2_t L^{q_n}_x}\leq C+C\,\|\vartheta_\eps\|_{L^2_t L^{\bar\alpha}_x}^{\bar\alpha\,\frac{\gamma-2n}{2\,\gamma}}. 
\]
This lets us deduce that for some $\beta<1/2$,
\[
\sup_t \int_{\TT^d} g_\eps(t,x)\,dx\leq C(\|\rho\|_{L^\infty_t L^\gamma_x},\|u\|_{L^2_tH^1_x})\,(1+\|\vartheta_\eps\|_{L^2_t L^{\bar\alpha}_x}^{\beta\,\bar\alpha}).
\]
For further use in a later section, we also note that we have the more precise estimate
\begin{equation}
\int_0^T\int_{\TT^d} |\tilde P_\eps|\,|\div u_\eps|\,dx\,dt\leq \|u\|_{L^2_tH^1_x}\,\left(\frac{\|\rho\|_{L^\infty_t L^\gamma_x}^{\gamma/2}}{C}+C\,\|\vartheta_\eps\|_{L^2_t L^{\bar\alpha}_x}^{\bar\alpha/2}\right).\label{Pdivu}
\end{equation}
From the definition of $g_\eps$, we also have that
\[
\int_{\TT^d} g_\eps(t,x)\,dx\geq \frac{1}{C}\,\int \vartheta^{\gamma_\theta}\,dx -\sum_{n=2}^N \|\vartheta_\eps^2\,B_n'(\vartheta_\eps)\,\rho_\eps^n\|_{L^\infty_t L^1_x}. 
\]
Using again the $L^\infty_t L^\gamma_x$ bound on $\rho$, this implies that
\[
\sum_{n=2}^N \|\vartheta_\eps^2\,B_n'(\vartheta_\eps)\,\rho_\eps^n\|_{L^\infty_t L^1_x}\leq \sum_{n=0}^N \|\vartheta_\eps^2\,B_n'(\vartheta_\eps)\|_{L^\infty_t L^{\tilde q_n}_x}\,\|\rho_\eps\|_{L^\infty_{t} L^\gamma_x}^{n},
\]
with $1/\tilde q_n=1-n/\gamma$. 

Now use assumption~\eqref{P6} to obtain, again for some $\tilde\beta<1$,
\[
\|\vartheta_\eps^2\,B_n'(\vartheta_\eps)\|_{L^\infty_t L^{\tilde q_n}_x}\leq C\,\|\vartheta_\eps\|_{L^\infty_t L^{\gamma_\vartheta}_x}^{\tilde\beta\,\gamma_{\vartheta}\,(\gamma-n)/\gamma}. 
\]
This shows that, again for some $\tilde \beta<1$,
\[
\sup_t \int_{\TT^d} g_\eps(t,x)\,dx\geq  \frac{1}{C}\,\sup_t\,\int_{\TT^d} \vartheta_\eps^{\gamma_\theta}\,dx- C\|\rho\|_{L^\infty_t L^\gamma_x}^\gamma\, -C\|\vartheta_\eps\|_{L^\infty_t L^{\gamma_\vartheta}_x}^{\tilde\beta\,\gamma_{\vartheta}}.
\]
For further use, we even have the more precise estimate
\begin{equation}
\sup_t \int_{\TT^d} g_\eps(t,x)\,dx\geq \frac{1}{C}\,\sup_t\,\int_{\TT^d} \vartheta_\eps^{\gamma_\theta}\,dx-C-\frac{1}{C}\,\sup_t\int_{\TT^d} \rho^\gamma\,dx.\label{lowerboundg}
  \end{equation}
Therefore, inserting those estimates into~\eqref{apriorigeps} yields that 
\begin{equation}\label{supttheta}
\sup_t\,\int_{\TT^d} \vartheta_\eps^{\gamma_\theta}\,dx\leq C+C\,\sup_t \int_{\TT^d} g_\eps(t,x)\,dx\leq C+C(\|\rho\|_{L^\infty_t L^\gamma_x},\|u\|_{L^2_tH^1_x})\,(1+\|\vartheta_\eps\|_{L^2_t L^{\bar\alpha}_x}^{\beta\,\bar\alpha}). 
\end{equation}
On its own, we cannot obtain a priori estimates just from~\eqref{supttheta} and we need also an entropy bound.
  Since for $\eps>0$, $g_\eps$ and $\vartheta_\eps$ are smooth and $\vartheta_\eps>0$, we can define
  \[
s_\eps=\frac{\eps}{\rho_\eps}\,\log \vartheta_\eps-\sum_{n=2}^N \tilde B_n(\vartheta_\eps)\,\frac{\rho_\eps^{n-1}}{n-1}+\frac{1}{\rho_\eps}\,\tilde B_0(\vartheta_\eps), 
\]
with $d \tilde B_n(\vartheta)/d\vartheta=\vartheta\,B_n''(\vartheta)+2\,B_n'(\vartheta)$ so that $\vartheta_\eps\,\frac{\partial s_\eps}{\partial \vartheta}=\frac{\partial \tilde e_\eps}{\partial \vartheta}$.

  Note that
  \[
\partial_t g_\eps=\frac{\partial g_\eps}{\partial\vartheta}\,\partial_t\vartheta_\eps+\frac{\partial g_\eps}{\partial\rho}\,\partial_t\rho_\eps,
  \]
  where we recall that
  \[
\frac{\partial g_\eps}{\partial\vartheta}=\eps-\sum_{2\le n\le N}\frac{d}{d \vartheta}(\vartheta\,^2B_n'(\vartheta_\eps))\frac{\rho^n}{n-1} + \frac{d}{d \vartheta}(\vartheta^2B_0'(\vartheta_\eps))>0,
\]
for any $\eps>0$ and for $g_\eps=\rho_\eps\,\tilde e_\eps$,
\[
\frac{\partial \tilde e_\eps}{\partial\rho}=-\eps\,\frac{\vartheta_\eps}{\rho_\eps^2}-\sum_{n=2}^N \vartheta_\eps^2\,B_n'(\vartheta_\eps)\,\rho_\eps^{n-2} -\vartheta_\eps^2\,B_0'(\vartheta_\eps)\,\frac{1}{\rho_\eps^2}.
\]
This lets us write that
\[\begin{split}
&\frac{\partial g_\eps}{\partial\vartheta}\,(\partial_t \vartheta_\eps+u_\eps\cdot \nabla_x\vartheta_\eps)=-g_\eps\,\div u_\eps-\frac{\partial g_\eps}{\partial\rho}\,(\partial_t \rho_\eps+u_\eps\cdot \nabla_x\rho_\eps)-\tilde P_\eps\,\div u_\eps\\
&\qquad\qquad+\cs_\eps:\nabla u_\eps+\nabla_x(\kappa(\vartheta_\eps)\,\nabla_x\vartheta_\eps).
\end{split}
\]
Since we have kept the critical relation,
\[
\tilde P_\eps=\rho^2\,\frac{\partial \tilde e_\eps}{\partial\rho}+\vartheta\,\frac{\partial \tilde P_\eps}{\partial\vartheta},
\]
this yields
\[\begin{split}
&\frac{\partial g_\eps}{\partial\vartheta}\,(\partial_t \vartheta_\eps+u_\eps\cdot \nabla_x\vartheta_\eps)=-\frac{\partial g_\eps}{\partial\rho}\,R_\eps-\vartheta_\eps\,\frac{\partial\tilde P_\eps}{\partial \vartheta}\,\div u_\eps\\
&\qquad\qquad+\cs_\eps:\nabla u_\eps+\nabla_x(\kappa(\vartheta_\eps)\,\nabla_x\vartheta_\eps).
\end{split}
\]
Because the critical relation above also implies that
\[
\frac{\partial s_\eps}{\partial\rho}=-\frac{1}{\rho^2}\,\frac{\partial \tilde P_\eps}{\partial\vartheta},
\]
we finally deduce that
\begin{equation}
  \begin{split}
  &\partial_t(\rho_\eps\,s_\eps)+\div(\rho_\eps\,s_\eps\,u_\eps)= \nabla_x\left(\frac{\kappa(\vartheta_\eps)}{\vartheta_\eps}\,\nabla_x\vartheta_\eps\right)\\
&\qquad\qquad    R_\eps\,\left(s_\eps+\rho_\eps\,\frac{\partial s_\eps}{\partial\rho}-\frac{1}{\vartheta_\eps}\,\frac{\partial g_\eps}{\partial\rho}\right)+\frac{1}{\vartheta_\eps}\,\cs_\eps:\nabla u_\eps+\kappa(\vartheta_\eps)\,\frac{|\nabla \vartheta_\eps|^2}{\vartheta_\eps^2}.
\end{split}\label{eqseps}
  \end{equation}
The first key point to make use of \eqref{eqseps} is that we still have that $\rho_\eps\,s_\eps\leq C\,g_\eps$. Hence
\[
\begin{split}
  \int_0^T \int_{\TT^d} \kappa(\vartheta_\eps)\,\frac{|\nabla \vartheta_\eps|^2}{\vartheta_\eps^2}\,dx\,dt\leq &C\,\int_{\TT^d} g_\eps(t=T,x)\,dx-\int_{\TT^d} \rho_\eps^0\,s_\eps^0(x)\,dx\\
  &+\|R_\eps\|_{L^2_t L^p_x}\,\left\|s_\eps+\rho_\eps\,\frac{\partial s_\eps}{\partial\rho}-\frac{1}{\vartheta_\eps}\,\frac{\partial g_\eps}{\partial\rho}\right\|_{L^2_t L^{p^*}_x}.
\end{split}
\]
We may hence immediately use~\eqref{supttheta} together with the $H^1$ estimate on $u_\eps$. We note that, just as for $g_\eps^0$, we have initial uniform bounds on  $\rho_\eps^0\,s^0_\eps$ in $L^1$. Indeed since $\vartheta_\eps^0=\vartheta^0+\eps$, and $\vartheta^0\in L^{\gamma_\vartheta}$ then
\[
\|\eps\,\log \vartheta^0_\eps\|_{L^1}\leq \eps\,|\log \eps|+\eps\,\|\vartheta^0\|_{L^1}\to 0,\quad \mbox{as}\ \eps\to 0.
\]
Moreover, again using assumption~\eqref{P6}, we also have the strong convergence in $L^1$ of $\tilde B_n(\vartheta_\eps^0)\,(\rho_\eps^0)^{n}$ just as for $g_\eps^0$. Consequently we have the strong convergence in $L^1_x$ of $\rho_\eps^0\,s_\eps^0$ to $\rho^0\,s^0$ with a uniform bound in $\eps$ which allows to derive
\begin{equation}
\begin{split}
  \int_0^T \int_{\TT^d} \kappa(\vartheta_\eps)\,\frac{|\nabla \vartheta_\eps|^2}{\vartheta_\eps^2}\,dx\,dt\leq &C+\|u\|_{L^2_t H^1_x}^2+C(\|\rho\|_{L^\infty_t L^\gamma_x},\|u\|_{L^2_tH^1_x})\,(1+\|\vartheta_\eps\|_{L^2_t L^{\bar\alpha}_x}^{\beta\,\bar\alpha})\\
  &+\|R_\eps\|_{L^2_t L^p_x}\,\left\|s_\eps+\rho_\eps\,\frac{\partial s_\eps}{\partial\rho}-\frac{1}{\vartheta_\eps}\,\frac{\partial g_\eps}{\partial\rho}\right\|_{L^2_t L^{p^*}_x}.
\end{split}\label{combinedapriori}
  \end{equation}
Our second critical point is that we have a simplified expression
\[
s_\eps+\rho_\eps\,\frac{\partial s_\eps}{\partial\rho}-\frac{1}{\vartheta_\eps}\,\frac{\partial g_\eps}{\partial\rho}=\sum_{n=2}^N (\vartheta_\eps\,B_n'-\tilde B_n)\,\frac{n}{n-1}\,\rho_\eps^{n-1},
\]
where the $\log \vartheta_\eps$ term vanish. In particular this expression is smooth around $\vartheta_\eps=0$ and only blows up as $\vartheta_\eps\to\infty$.

Recalling the definition of $\tilde B_n$, we note that
\[
\frac{d}{d\vartheta} (\vartheta\,B_n'-\tilde B_n)=B_n'+\vartheta\,B_n''-\vartheta\,B_n''-2B_n'=-B_n',
\]
so that in the end
\[
s_\eps+\rho_\eps\,\frac{\partial s_\eps}{\partial\rho}-\frac{1}{\vartheta_\eps}\,\frac{\partial g_\eps}{\partial\rho}=\sum_{n=2}^N B_n(\vartheta_\eps)\,\frac{n}{n-1}\,\rho_\eps^{n-1}.
\]
Using the $L^\infty_t L^\gamma_x$ bound on $\rho$, this implies that
\[
\left\|s_\eps+\rho_\eps\,\frac{\partial s_\eps}{\partial\rho}-\frac{1}{\vartheta_\eps}\,\frac{\partial g_\eps}{\partial\rho}\right\|_{L^2_t L^{p^*}_x}\leq C\, \sum_{n=2}^N \left\|B_n(\vartheta_\eps)\,\rho_\eps^{n-1}\right\|_{L^2_t L^{p^*}_x}\leq C\,\sum_{n=2}^N \|\rho_\eps\|_{L^\infty_t L^\gamma_x}^{n-1}\,\left\|B_n(\vartheta_\eps)\right\|_{L^2_t L^{p_n}_x},
\]
with
\[
p_n=\frac{1}{p^*}-\frac{n-1}{\gamma}=1-\frac{1}{p}-\frac{n-1}{\gamma} =\frac{1}{2}-\frac{1}{\gamma}-\frac{n-1}{\gamma} =\frac{1}{2}-\frac{n}{\gamma}=q_n, 
\]
which was the exponent defined earlier when controlling the norm of $\tilde P_\eps$ in $L^2$.

Therefore this term can be bounded in exactly the same way through assumption~\eqref{P6bis}, yielding
\[
\left\|s_\eps+\rho_\eps\,\frac{\partial s_\eps}{\partial\rho}-\frac{1}{\vartheta_\eps}\,\frac{\partial g_\eps}{\partial\rho}\right\|_{L^2_t L^{p^*}_x}\leq C+C(\|\rho\|_{L^\infty_tL^\gamma_x})\, (1+\|\vartheta_\eps\|_{L^2_t L^{\bar\alpha}_x}^{\beta\,\bar\alpha}).
\]
This let us obtain by adding \eqref{combinedapriori} and \eqref{supttheta} that
\begin{equation}
\sup_t \int_{\TT^d} \vartheta_\eps^{\gamma_\vartheta}\,dx + \int_0^T \int_{\TT^d} \kappa(\vartheta_\eps)\,\frac{|\nabla \vartheta_\eps|^2}{\vartheta_\eps^2}\,dx\,dt\leq C+C(\|\rho\|_{L^\infty_t L^\gamma_x},\|u\|_{L^2_tH^1_x})\,(1+\|\vartheta_\eps\|_{L^2_t L^{\bar\alpha}_x}^{\beta\,\bar\alpha}).
\label{combinedapriori2}
  \end{equation}
We are now ready to conclude our a priori estimates. By assumption~\eqref{con:ass} and Poincar\'e inequalities,
\[
\int_{\TT^d} \vartheta_\eps^{\bar\alpha}\,dx\leq \left(\int_{\TT^d} \vartheta_\eps^{\gamma_\vartheta}\,dx\right)^{\bar\alpha/\gamma_\vartheta}+\int_{\TT^d} \kappa(\vartheta_\eps)\,\frac{|\nabla \vartheta_\eps|^2}{\vartheta_\eps^2}\,dx,
\]
so that, since $\alpha\geq 2$,
\begin{equation}
\|\vartheta_\eps\|_{L^2_t L^{\bar\alpha}_x}^{\beta\,\bar\alpha}\leq C\,T^{\beta\,\bar\alpha/2}\,\left(\sup_t \int_{\TT^d} \vartheta_\eps^{\gamma_\vartheta}\,dx\right)^{\beta\,\bar\alpha/\gamma_\vartheta}+C_T\,\left(\int_0^T \int_{\TT^d} \kappa(\vartheta_\eps)\,\frac{|\nabla \vartheta_\eps|^2}{\vartheta_\eps^2}\,dx\,dt\right)^\beta.\label{Poincare}
\end{equation}
Because $\beta<1/2$ and $\bar\alpha\leq 2\,\gamma_\vartheta$, this inequality is enough to bound the right-hand side of \eqref{combinedapriori2} in terms of its left-hand side. Hence we eventually have the uniform in $\eps$ estimates
\begin{equation}
\sup_\eps\sup_t \int_{\TT^d} \vartheta_\eps^{\gamma_\vartheta}\,dx +\sup_\eps \int_0^T \int_{\TT^d} \kappa(\vartheta_\eps)\,\frac{|\nabla \vartheta_\eps|^2}{\vartheta_\eps^2}\,dx\,dt\leq C(\|\rho\|_{L^\infty_t L^\gamma_x},\|u\|_{L^2_tH^1_x}).\label{aprioriepsuniform}
\end{equation}
Those bounds directly imply that $\vartheta\in L^\infty([0,\ T],\;L^{\gamma_\vartheta}(\Pi^d))\cap L^\alpha([0,\ T]\times\Pi^d)$ as claimed. 
Because $\kappa(\vartheta)\geq 1$, \eqref{aprioriepsuniform} also shows that
\begin{equation}
  \begin{split}
    &\sup_\eps \int_0^T \int_{\TT^d}  |\nabla \log \vartheta_\eps|^2\,dx\,dt\leq C(\|\rho\|_{L^\infty_t L^\gamma_x},\|u\|_{L^2_tH^1_x}),\\
    &\mbox{and}\ \sup_\eps \int_0^T \int_{\TT^d} |\log \vartheta_\eps|^2\,dx\,dt\leq C(\|\rho\|_{L^\infty_t L^\gamma_x},\|u\|_{L^2_tH^1_x}),
  \end{split}
  \label{aprioriepsuniform2}
\end{equation}
again by Poincar\'e inequality.

To conclude those a priori estimates, note that we finally have that
\begin{equation}
\sup_\eps\int_0^T \int _{\TT^d}\frac{|\nabla u|^2}{\vartheta}\,dx\,dt\leq C(\|\rho\|_{L^\infty_t L^\gamma_x},\|u\|_{L^2_tH^1_x}).\label{aprioriepsuniform3}
  \end{equation}

\medskip

\noindent {\it Limit passage $\varepsilon \to 0$.}
We can now send $\eps\rightarrow0$ to get a weak solution of~\eqref{div:form:g}.
From our previous estimates, we know that $g_\eps$ is uniformly bounded in $L^\infty_t L^p_x$ for some $p>1$. This lets us extract a sub-sequence, still denoted $g_\eps$, that converges weak-* to some $g$ in $L^\infty_t L^p_x$ for some $p>1$.

To derive the compactness on $\vartheta_\eps$ through the classical Aubin-Lions approach, we require controls on $g_\eps\,u_\eps$ and $\rho_\eps\,s_\eps\,u_\eps$. We may bound directly by Sobolev embeddings
\[
\|g_\eps\,u_\eps\|_{L^1_{t,x}}\leq \|u_\eps\|_{L^2_t H^1_x}\,\|g_\eps\|_{L^2_{t,x}}.
\]
It is straightforward to bound the $L^2$ norm of $g_\eps$ in the same manner as we bounded the $L^2$ norm of $\tilde P_\eps$ earlier: Assumption~\eqref{P6bis} indeed implies the same behavior for $\vartheta^2\,B_n'(\vartheta)$ and $\vartheta\,B_n(\vartheta)$.

For further use, we also observe that by using the $\eps>0$ in \eqref{P6bis}, we may use some interpolation on $\vartheta_\eps$ between $L^\infty_t L^{\gamma_\vartheta}_x$ and $L^2_t L^{\bar \alpha}_x$, leading actually to
\begin{equation}
  \|g_\eps\,u_\eps\|_{L^p_{t,x}}\leq C(\|\rho\|_{L^\infty_t L^\gamma_x},\|u\|_{L^2_tH^1_x}),
  \label{Lpflux1}
\end{equation}
for some $p>1$. The same applies to $\tilde P_\eps$ so that we also have that
\begin{equation}
  \|\tilde P_\eps\|_{L^q_{t,x}}\leq C(\|\rho\|_{L^\infty_t L^\gamma_x},\|u\|_{L^2_tH^1_x}),
  \label{LqPeps}
\end{equation}
for some $q>1$.

A similar discussion applies to $\rho_\eps\,s_\eps$, with in fact much simpler estimates. First of all $\tilde B_n$ behaves like $\vartheta\, B_n''+2\,B_n'$ instead of $\vartheta^2\,B_n'(\vartheta)$ so that the coefficients $n\geq 2$ in the expansion are easier to handle than for $g_\eps$. Secondly, the $\eps\,\log \vartheta_\eps$ in $\rho_\eps\,s_\eps$ is immediately bounded by \eqref{aprioriepsuniform2}. Hence we also have that
\begin{equation}
  \|\rho_\eps\, s_\eps\,u_\eps\|_{L^p_{t,x}}\leq C(\|\rho\|_{L^\infty_t L^\gamma_x},\|u\|_{L^2_tH^1_x}),
  \label{Lpflux2}
\end{equation}
for some $p>1$.

We now turn to the compactness argument. We may extract a subsequence $\vartheta_\eps$, converging weak-* to $\vartheta$ in $L^\infty_t L^{\gamma_\vartheta}_x$. Furthermore by~\eqref{aprioriepsuniform} it follows that $\vartheta_\eps$ is compact in space. Since $\rho_\eps$ is also compact in space, the definition~\eqref{gtothetaeps} of $g_\eps$ together with our a priori estimates directly implies that $g_\eps$ is compact in space. For similar reasons, $\rho_\eps\,s_\eps$ is compact in space.

We now obtain from Equation~\eqref{app} and Equation~\eqref{eqseps}  that both $\partial_tg_\eps$ and $\partial_t(\rho_\eps\,s_\eps)$ are bounded in $L^1_t W^{-1,1}_x$ thanks to \eqref{Lpflux1}-\eqref{Lpflux2} and our previous a priori estimates. By Aubin-Lions, this shows that $g_\eps$ and $\rho_\eps\,s_\eps$ are compact in $L^1_{t,x}$.

Upon further extraction, we may therefore assume that both $g_\eps$ and $\rho_\eps\,s_\eps$ converge pointwise $a.e.$ respectively to $g$ and some $S$. Of course $\rho_\eps$ converges $a.e.$ to $\rho$. By assumptions~\eqref{P2} and \eqref{P5}, $\partial_{\vartheta_\eps} g_\eps\geq 0$ and more precisely $\partial_{\vartheta_\eps} g_\eps$ is uniformly away from $0$ for $\vartheta_\eps>\underline{\vartheta}$ for any $\underline{\vartheta}>0$. This proves that for a fixed value of $\rho_\eps(t,x)$, $g_\eps=g_\eps(\rho_\eps,\vartheta_\eps)$ is one-to-one in $\vartheta_\eps$.

The pointwise convergence of $g_\eps$ therefore implies the pointwise convergence of $\vartheta_\eps$ to some $\vartheta$, and hence the compactness and convergence of $\vartheta_\eps$ to $\vartheta$ in $L^1_{t,x}$. A first consequence is that we may pass to the limit in~\eqref{gtothetaeps} and obtain that the limits $\rho$, $\vartheta$ and $g$ solve~\eqref{new:e}.  Similarly $\rho$, $\vartheta$ and $\rho\,s$ solve~\eqref{defentropy}.

\medskip

\noindent {\it Energy equation \eqref{energyequality}.}
It remains to pass to the limit in the integral of Equation~\eqref{app} on $g_\eps$ and in Equation~\eqref{eqseps}. Since $u_\eps$ is converging  $a.e.$ to $u$, we have the $a.e.$ convergence of $g_\eps\,u_\eps$ and $\rho_\eps\,s_\eps\,u_\eps$ to respectively $g\,u$ and $\rho\,s\,u$. By the equi-integrability provided by~\eqref{Lpflux1} and~\eqref{Lpflux2}, we can apply dominated convergence and obtain the strong convergence of $g_\eps\,u_\eps$ and $\rho_\eps\,s_\eps\,u_\eps$.

Obviously we directly have the strong convergence of $\cs_\eps:\nabla u_\eps$. We also have pointwise convergence inside the formula~\eqref{Peps} defining $\tilde P_\eps$ so that $\rho$, $\vartheta$ and $\tilde P$ satisfy~\eqref{new:p}. By~\eqref{LqPeps}, we hence have that $\tilde P_\eps$ converges to $\tilde P$ in $L^2_{t,x}$, again by dominated convergence. Since $\div u_\eps$ converges strongly to $\div u$ in $L^2_{t,x}$, this yields the convergence of $\tilde P_\eps\,\div u_\eps$.
It is now possible to integrate Equation~\eqref{app} and pass to the limit in all resulting terms to obtain the claimed energy equality~\eqref{energyequality}.  

\noindent {\it Entropy inequation \eqref{eqslim}.}
It remains to derive the limit of Equation~\eqref{eqseps} on $\rho_\eps\,s_\eps$. Our previous analysis shows that
\[
R_\eps\,\left(s_\eps+\rho_\eps\,\frac{\partial s_\eps}{\partial\rho}-\frac{1}{\vartheta_\eps}\,\frac{\partial g_\eps}{\partial\rho}\right)\longrightarrow 0,
\]
strongly as $\eps\to 0$.

We can also prove 
that $\div\left(\frac{\kappa(\vartheta_\eps)}{\vartheta_\eps}\,\nabla\vartheta_\eps \right) \rightarrow \div\left(\frac{\kappa(\vartheta)}{\vartheta}\,\nabla\vartheta\right)$ in the distribution sense. Denoting $\tilde \kappa(\vartheta)$ s.t. $\tilde\kappa'=\frac{\kappa(\vartheta)}{\vartheta}$, we note that
\[
\frac{\kappa(\vartheta_\eps)}{\vartheta_\eps}\,\nabla\vartheta_\eps=\nabla(\tilde \kappa(\vartheta_\eps)).
\]
As before $\tilde \kappa(\vartheta_\eps)$ converges $a.e.$ to $\tilde \kappa(\vartheta)$. By assumption~\eqref{con:ass}, $|\tilde \kappa(\vartheta)|\leq C\,(\log\vartheta+\vartheta^{\alpha})$. On the other hand, by combining~\eqref{con:ass} and~\eqref{aprioriepsuniform}, we also have that
\[
\sup_\eps \int_0^T \int_{\TT^d} \vartheta_\eps^{\alpha-2}\,|\nabla\vartheta_\eps|^2\,dx\,dt=\sup_\eps \int_0^T \int_{\TT^d} |\nabla\vartheta_\eps^{\alpha/2}|^2\,dx\,dt. 
\]
By Sobolev embedding, we have that $\vartheta_\eps$ is uniformly bounded in $L^\alpha_t L^{\alpha\,2^*/2}_x$ with $1/2^*=1/2-1/d$ (or $2^*<\infty$ for $d=2$). By interpolation with the uniform bound in $L^{\infty}_t L^{\gamma_\vartheta}_x$, we obtain a uniform bound for $\vartheta_\eps$ in $L^p_{t,x}$ for some $p>\alpha$. As a consequence $\tilde \kappa(\vartheta_\eps)$ is equi-integrable and converges strongly in $L^1_{t,x}$ to $\tilde \kappa(\vartheta)$, proving the required limit.

It is important to highlight that the same argument would not apply to 
the limit of $\div\left(\kappa(\vartheta_\eps)\,\nabla\vartheta_\eps \right)$. Any anti-derivative of $\kappa(\vartheta)$ behaves like $\vartheta^{\alpha+1}$ as $\vartheta\to \infty$. Therefore it would not in general be possible to control it through our a priori estimates. This is the main objection that prevents us from passing to the limit in the whole equation~\eqref{app} for $g_\eps$.

We are also not able to pass to the limit in the two remaining terms in the right-hand side of Equation~\eqref{eqseps}. We have for example the $a.e.$ convergence of $\frac{1}{\vartheta_\eps}\,\cs_\eps:\nabla u_\eps$ but we cannot prove equi-integrability, as $\frac{1}{\vartheta_\eps}$ could be large. However we can obtain inequalities which lead to the limiting inequation.

We recall that the function $(a,b)\to \frac{a^2}{b}$ is jointly convex in $(a,b)$. Consequently if some functions $a_n,\,b_n$ converge to functions $a,\,b$ in $L^1$ (or even in some appropriate weak topology) then $\frac{a^2}{b}\leq \liminf \frac{a^2_n}{b_n}$.  This immediately implies that
\[
\frac{1}{\vartheta}\,\cs:\nabla u\leq \liminf_{\eps\to 0}\frac{1}{\vartheta_\eps}\,\cs_\eps:\nabla u_\eps.
\]
Second by denoting $\bar\kappa(\vartheta)$ s.t. $\bar\kappa'=(\kappa(\vartheta))^{1/2}/\vartheta$, we have that
\[
\kappa(\vartheta_\eps)\,\frac{|\nabla\vartheta_\eps|^2}{\vartheta_\eps^2} =|\nabla\bar\kappa(\vartheta_\eps)|^2.
\]
Therefore, we also have that
\[
\kappa(\vartheta)\,\frac{|\nabla\vartheta|^2}{\vartheta^2}\leq \liminf_{\eps\to 0}\kappa(\vartheta_\eps)\,\frac{|\nabla\vartheta_\eps|^2}{\vartheta_\eps^2}.
\]
The same arguments allow us to deduce~\eqref{aprioritheta} from our a priori estimates~\eqref{aprioriepsuniform}-\eqref{aprioriepsuniform3}. Concerning 
\eqref{inis}, we use that
$$\tau \in [0,T] \mapsto \int_\Omega (\rho s)(t,\cdot) \varphi dx 
\qquad \varphi \in {\mathcal C}^1_c(\Omega) \hbox{ and } \varphi\ge 0$$
is a sum of a non-decreasing function and a continuous function taking
advantage of the entropy inequality. This completes the proof.
\end{proof}
\section{Third Step: Fixed Point procedure and proof of main result} 
We are now ready to prove our main result. Denote
\[
E=L^{\alpha-\eps'}([0,T]\times \Pi^d)\cap L^1([0,\ T],\;W^{\lambda,1}(\Pi^d)),
\]
with $\eps'>0$ as in Theorem~\ref{existrhou} and any $0<\lambda<1$. For any $R>0$, denote as well
\[
E_R=\{\vartheta\in E\,|\;\|\vartheta\|_{L^{\alpha-\eps'}_{t,x}}+\|\vartheta\|_{L^1_t W^{\lambda,1}_{x}}\leq R\}.
\]
We now define the operator $L$ on $E$ that will have a fixed point. For a given $\vartheta_i$ in $E$, we may use Theorem~\ref{existrhou} to obtain  solutions $\rho$ and $u$ to \eqref{eq:00}-\eqref{eq:01}, and satisfying the estimates
\[
\rho\in L^\infty_t L^\gamma_x,\quad u\in L^2_t H^1_x.
\]
Hence $\rho$ and $u$ satisfy all the conditions in Theorem~\ref{solvevartheta}. We may hence apply Theorem~\ref{solvevartheta} to obtain $\vartheta=L(\vartheta_i)$ that solves the various estimates listed in the statement of Theorem~\ref{solvevartheta}. In particular by~\eqref{aprioritheta}, we immediately have,  from the bounds~\eqref{con:ass} on $\kappa$, that
\[
\int_0^T \int_{\TT^d}
 (1+|\vartheta|^\alpha)\,\frac{|\nabla_x\vartheta|^2}{\vartheta^2}\,dx\,dt<\infty. 
\]
Since $\vartheta\in L^\infty_t L^{\gamma_\vartheta}_x$ as well, Poincar\'e inequality immediately shows that $\vartheta\in L^\alpha_{t,x}$. Moreover we also have  directly from the inequality above that $\vartheta\in L^2_t H^1_x$. Hence for any $0<\lambda<1$, we have that $\vartheta\in L^1_t W^{\lambda,1}_x$.  This implies that $L:\, E\to E$.

\bigskip

We can also check that for any $R>0$, the image $L(E_R)$ is pre-compact in $E$. 
Consider therefore any sequence $\vartheta_n\in E_R$. 
From the estimates in Theorem~\ref{existrhou}, we have that for some $C_R$ and the $\rho_n,\,u_n$ obtained from $\vartheta_n$ satisfy the uniform bound 
\[
\sup_n \|\rho_n\|_{L^\infty_t L^\gamma_x}+\sup_n \|u_n\|_{L^2_t H^1_x}\leq C_R.
\]
Moreover Theorem~\ref{existrhou} also implies that $\rho_n$ is compact in $L^1_{t,x}$.

Consequently, the estimates~\eqref{aprioritheta} from Theorem~\ref{solvevartheta} yields, also for some $C_R$
\[
\sup_n \|L(\vartheta_n)\|_{L^\infty_t L^{\gamma_{\vartheta}}}
+\sup_n \int_0^T \int_{\TT^d}
  (1+|L(\vartheta_n)|^\alpha)\,\frac{|\nabla_xL(\vartheta_n)|^2}{L(\vartheta_n)^2}\,dx\,dt \leq C_R.
\]
We next observe that the entropy inequation~\eqref{eqslim} provides a uniform control on $\partial_t (\rho_n s_n)$. We may indeed rewrite~\eqref{eqslim} as
\[
\begin{split}
  \partial_t(\rho_n\,s_n)+\div(\rho_n\,s_n\,u_n)+M_n(t,x)= &\nabla_x\left(\frac{\kappa(L(\vartheta_n))}{L(\vartheta_n)}\, \nabla_xL(\vartheta_n)\right) +\frac{1}{L(\vartheta_n)}\,\cs_n:\nabla u_n\\
  &+\kappa(L(\vartheta_n))\,\frac{|\nabla L(\vartheta_n)|^2}{L(\vartheta_n)^2},
\end{split}
\]
where $s_n=s(\rho_n,L(\vartheta_n))$ and $M_n$ is a sequence of non-negative Radon measures. Hence by integrating in $t$ and $x$, we have the bound
\[\begin{split}
\int_0^T\int_{\TT^d} M(dt,dx)&=\int_{\TT^d} \rho^0\,s(\rho^0,\vartheta^0)-\int_{\TT^d}  \rho_n\,s_n|_{t=T}\,dx+\int_0^T\int_{\TT^d} \frac{1}{\vartheta_n}\,\cs_n:\nabla u_n\,dx\,dt\\
&+\int_0^T\int_{\TT^d}\kappa(L(\vartheta_n))\,\frac{|\nabla L(\vartheta_n)|^2}{L(\vartheta_n)}^2\,dx\,dt.
\end{split}
\]
From~\eqref{entropybounded} and \eqref{aprioritheta}, we deduce immediately that the total mass of $M_n$, as a measure in $t$ and $x$, is uniformly bounded in $n$. Using again \eqref{entropybounded} and \eqref{aprioritheta}, this implies that $\partial_t (\rho_n s_n)$ is uniformly bounded in $\mathcal{M}_{t,x}+L^1_t W^{-1,1}_x$, with $\mathcal{M}_{t,x}$ the set of Radon measures with bounded mass.

From the compactness of $\rho_n$, the compactness in space of $L(\vartheta_n)$, we have compactness in space for $\rho_n s_n$ and now compactness in time. Up to extracting a subsequence, we can hence deduce the pointwise convergence of $\rho_n s_n$. Following the same argument as in the proof of Theorem~\ref{solvevartheta}, and in particular assumption~\eqref{P7}, this yields the pointwise convergence of $L(\vartheta_n)$. 

From the uniform bounds on $L(\vartheta_n)$ in $L^\alpha_{t,x}$, this in turn implies the compactness of $L(\vartheta_n)$ in $L^{\alpha-\eps}_{t,x}$. By interpolation between $L^{\alpha-\eps}_{t,x}$ and $L^2_t H^1_x$, we also obtain compactness in $L^1_t W^{\lambda,1}_x$, showing that the image $L(E_R)$ is pre-compact.

\bigskip

The last and more delicate point to use the Leray-Schauder fixed point is to show that there exists $R$ s.t. for any $\vartheta\in E$ with $\vartheta=\ell\,L(\vartheta)$ for $\ell\in [0,\ 1]$, we have that $\vartheta\in E_R$. For such $\vartheta\in E$ with $\vartheta=\ell\,L(\vartheta)$, we start with recalling from Theorem~\ref{existrhou} that
\[\begin{split}
      &\int_{\TT^d} (\rho^\gamma+\rho\,|u|^2/2)\,dx+\int_0^t\int_{\TT^d} \cs:\nabla_x u\,dx\,ds\leq {\mathcal E}(\rho_0,u_0)\\
      &\qquad+\,\int_0^t\int_{\TT^d} \div_x u(s,x)\,(P(\vartheta(s,x),\rho(s,x))-P_0(\rho(s,x)))\,ds\,dx.\\
      \end{split}
\]
From the definition of $P_0$,
\[
P-P_0=\sum_{n=0}^N (B_n(\vartheta)-\bar B_n)\,\rho^n.
\]
Therefore the coefficients in $P-P_0$ behave in the same way as for $\tilde P_\eps$ that we had used before and we may use \eqref{Pdivu} with
\[
\int_0^t\int_{\TT^d} |P-P_0|\,|\div u |\,dx\,dt\leq \|u\|_{L^2_tH^1_x}\,\left(\frac{\|\rho\|_{L^\infty_t L^\gamma_x}^{\gamma/2}}{C}+C\,\|\vartheta\|_{L^2_t L^{\bar\alpha}_x}^{\bar\alpha/2}\right).
\]
This implies that
\[
\int_{\TT^d} (\rho^\gamma+\rho\,|u|^2/2)\,dx+\int_0^t\int_{\TT^d} |\nabla_x u|^2\,dx\,ds\leq C+C\,\|\vartheta\|_{L^2_t L^{\bar\alpha}_x}^{\bar\alpha}.
\]
We will use as an intermediary object the function
\begin{equation}\label{phi:v}
\phi_\vartheta(t)=C+ \int_0^t\int_{\TT^d} |\nabla_x u|^2(s,x)\,dx\,ds+\|\vartheta\|_{L^2([0,\ t],\ L^{\bar\alpha}(\Pi^d))}^{\bar\alpha}.
\end{equation}
It will be important to note that $\phi_\vartheta(t)$ is continuous in time for a fixed choice of $\vartheta$, even if it is of course not equi-continuous for all possible choices of $\vartheta$.
On the other hand since $\vartheta=\ell\,L(\vartheta)$, any norm of $\vartheta$ is bounded by the corresponding norm of $L(\vartheta)$. Using the Poincar\'e inequality~\eqref{Poincare}, and since $\bar \alpha<\alpha$, this proves that
\[
\|\vartheta\|_{L^2_t L^{\bar\alpha}_x}^{\bar\alpha}\leq C\,T^{\beta\,\bar\alpha/2}\,\left(\sup_{t\le T}\int_{\TT^d}L(\vartheta)^{\gamma_\vartheta}\,dx\right)^{\beta\,\bar\alpha/\gamma_\vartheta}+C_T\,\left(\int_0^T \int_{\TT^d} \kappa(L(\vartheta))\,\frac{|\nabla L(\vartheta)|^2}{L(\vartheta)^2}\,dx\,dt\right)^\beta,
\]
for some $\beta<1$.

Therefore the norms of $\rho$ and $u$ together with $\phi_{\vartheta}$ are controlled through the corresponding norms of $L(\vartheta)$,
\begin{equation}
  \begin{split}
    &\int_{\TT^d} (\rho^\gamma+\rho\,|u|^2/2)\,dx+\int_0^t\int_{\TT^d} |\nabla_x u|^2\,dx\,ds\leq C\,\phi_\vartheta(t)\\
    &\ \leq C^2+C^2\,T^{\beta\,\bar\alpha/2}\,\left(\sup_t \int_{\TT^d} L(\vartheta)^{\gamma_\vartheta}\,dx\right)^{\beta\,\bar\alpha/\gamma_\vartheta}\\
    &\hskip2cm +C_T\,\left(\int_0^t \int_{\TT^d} \kappa(L(\vartheta))\,\frac{|\nabla L(\vartheta)|^2}{L(\vartheta)^2}\,dx\,dt\right)^\beta,\label{firstcoupledapriori}
\end{split}
  \end{equation}
where the constant $C$ depends only on the initial data and more precisely the initial total energy.
Turning to Theorem~\ref{solvevartheta}, we recall the important Energy equality \eqref{energyequality}
which implies that
\[
\int_{\TT^d} g(t,x)\,dx \leq C+\int_0^t \int_{\TT^d} (\cs:\nabla u-\tilde P\,\div u)\,dx\,ds
\]
for all time. Using again~\eqref{Pdivu} and the Poincar\'e inequality~\eqref{Poincare} together with~\eqref{firstcoupledapriori} to control $\int |\nabla u|^2$, we obtain that
\begin{equation}
\begin{split}
    &\sup_{t\leq T}\int_{\TT^d} g(t,x)\,dx\leq C\,\phi_{\vartheta}(T)\\
    &\quad \leq C^2+C^2\,T^{\beta\,\bar\alpha/2}\,\left(\sup_t \int_{\TT^d}  L(\vartheta)^{\gamma_\vartheta}\,dx\right)^{\beta\,\bar\alpha/\gamma_\vartheta} \\
    &\hskip2cm +C_T\,\left(\int_0^T \int_{\TT^d} \kappa(L(\vartheta))\,\frac{|\nabla L(\vartheta)|^2}{L(\vartheta)^2}\,dx\,dt\right)^\beta.
\end{split}\label{gLvartheta}
\end{equation}
From~\eqref{lowerboundg}, this shows that
\[
\begin{split}
    &\sup_{t\leq T}\int_{\TT^d}  (L\vartheta)^{\gamma_\vartheta}(t,x)\,dx\leq C\,\phi_{\vartheta}(T)\\
    &\quad \leq C^2+C^2\,T^{\beta\,\bar\alpha/2}\,\left(\sup_t \int_{\TT^d} L(\vartheta)^{\gamma_\vartheta}\,dx\right)^{\beta\,\bar\alpha/\gamma_\vartheta}+C_T\,\left(\int_0^T \int_{\TT^d} \kappa(L(\vartheta))\,\frac{|\nabla L(\vartheta)|^2}{L(\vartheta)^2}\,dx\,dt\right)^\beta.
\end{split}
\]
On the other hand, since $\rho\,s\leq C\,g$, we also obtain by integrating the inequation~\eqref{eqslim},
and by combining the result with~\eqref{gLvartheta} and the previous inequality, we finally obtain the critical estimate
\begin{equation}
\begin{split}
    &\sup_{t\leq T}\int_{\TT^d} (L\vartheta)^{\gamma_\vartheta}(t,x)\,dx+\int_0^T \int_{\TT^d} \kappa(L(\vartheta))\,\frac{|\nabla L(\vartheta)|^2}{L(\vartheta)^2}\,dx\,dt\leq C\,\phi_\vartheta(t)\\
    &\quad \leq C^2+C^2\,T^{\beta\,\bar\alpha/2}\,\left(\sup_{t\le T} \int_{\TT^d}  L(\vartheta)^{\gamma_\vartheta}\,dx\right)^{\beta\,\bar\alpha/\gamma_\vartheta} \\
    &\hskip2cm +C_T\,\left(\int_0^T \int_{\TT^d} \kappa(L(\vartheta))\,\frac{|\nabla L(\vartheta)|^2}{L(\vartheta)^2}\,dx\,dt\right)^\beta.
\end{split}\label{criticalvartheta}
\end{equation}
Of course $\beta<1$ but unfortunately we only have $\bar\alpha\leq 2\,\gamma_\vartheta$ so that we could have that $\beta\,\bar\alpha/\gamma_\vartheta>1$, which prevents us from concluding at once and forces us to employ a much more careful argument.
 The key point is to use the time continuity of $\phi_\vartheta$ defined in \eqref{phi:v}. Since $\beta<1$, denote
\[
M=\sup_{X>0} C_{T=1}\, X^\beta-X/2,
\]
and
\[
\Phi =2\,C+3\,C_{T=1}\,C^2+2\,C_{T=1}\,M.
\]
Assuming that $\beta\,\bar\alpha/\gamma_\vartheta>1$, we choose $T\leq 1$ s.t.
\[
T^{\beta\,\bar\alpha/2}\,(C\,\Phi)^{\beta\,\bar\alpha/\gamma_\vartheta}<\min(1/2,1/2C).
\]
From the continuity of $\phi_{\vartheta}(t)$ in time, we may define $t_0\leq T$ the largest time s.t. $\phi_{\vartheta}(t)\leq  \Phi$.
From \eqref{criticalvartheta}, we also have that  
\[
\sup_{t\le t_0} \int_{\TT^d}  (L\vartheta)^{\gamma_\vartheta}(t,x)\,dx\leq C\,\Phi.
\]
From using a second time~\eqref{criticalvartheta}, we deduce that
\[
\int_0^T \int_{\TT^d} \kappa(L(\vartheta))\,\frac{|\nabla L(\vartheta)|^2}{L(\vartheta)^2}\,dx\,dt\leq 2\,C^2+2\,C^2\,T^{\beta\,\bar\alpha/2}\,(C\,\Phi)^{\beta\,\bar\alpha/\gamma_\vartheta}+2\,M<3\,C^2+2\,M.
\]
However at $t=T$, taking again~\eqref{criticalvartheta} now implies that
\[
\begin{split}
  \phi_{\vartheta}(T)&\leq C+C\,T^{\beta\,\bar\alpha/2}\,(C\,\Phi)^{\beta\,\bar\alpha/\gamma_\vartheta-1}+C_{T=1}(3\,C^2+2\,M)^\beta\\
  &<2\,C+3\,C_{T=1}\,C^2+2\,C_{T=1}\,M=\Phi.
\end{split}
\]
This shows that $t_0=T$ and yields a corresponding bound on $L(\vartheta)$ in $L^\infty_t L^{\gamma_\vartheta}$ and in $L^\alpha_{t,x}$ in terms of the initial energy.
The same argument that we used at the beginning of the proof then show that $L(\vartheta)\in E_R$ for some $R$ depending only on the initial energy. Since $\vartheta=\ell\,L(\vartheta)$, we also have that $\vartheta\in E_R$ and we have checked all assumptions of the Leray-Schauder fixed point theorem. Consequently, {\em for this choice of T}, there exists a fixed point $\vartheta\in E$ s.t. $\vartheta=L(\vartheta)$.

\bigskip

We hence obtain a solution $(\rho,\,u,\,\vartheta)$ to \eqref{eq:00}, \eqref{eq:01} and the entropy inequation~\eqref{entropyineq} on $[0,\ T]$. By combining the estimates in Theorems~\ref{existrhou} and~\ref{solvevartheta}, we also recover all a priori estimates in Theorem~\ref{main} and it only remains to derive the global energy bound to have Theorem~\ref{main} on $[0,\ T]$. 
We first add the energy bounds~\eqref{energy0} given by Theorem~\ref{existrhou} and~\eqref{energyequality} given by Theorem~\ref{solvevartheta}.
\[\begin{split}
&\int_{\TT^d}  \mathcal{E}_0(\rho,u)\,dx+\int_{\TT^d}  g(t,x)\,dx+\int_0^t\int_{\TT^d} \cs:\nabla_x u\,dx\,ds\leq\mathcal{E}_0(\rho^0,u^0)\,dx+\int_{\TT^d}  g^0(x)\,dx\\
&\qquad\qquad+\int_0^t \int_{\TT^d}  (\cs:\nabla_x u+(P(\vartheta(s,x),\rho(s,x))-P_0(\rho(s,x))-\tilde P)\,\div u)(s,x)\,ds\,dx.
\end{split}
\]
By recombining the terms, we obtain that
\begin{equation}
\begin{split}
&\int_{\TT^d} \mathcal{E}(\rho,u,\vartheta)\,dx+\int_{\TT^d}  \left(\sum_{n\leq N} \bar B_n\,\frac{\rho^n}{n-1}\right)\,dx\leq  \int_{\TT^d}  \mathcal{E}(\rho^0,\vartheta^0,m_0)\,dx \\
& +\int_{\TT^d}  \left(\sum_{n\leq N} \bar B_n\,\frac{(\rho^0)^n}{n-1}\right)\,dx
-\int_0^t \int_{\TT^d}  \div u(s,x)\,\left(\sum_{n\leq N} \bar B_n\,\rho^n\right)\,dx\,ds,
\end{split}\label{almostenergy}
\end{equation}
as we can easily recognize the total energy $\mathcal{E}$ of the system.
 We also emphasize that it was critical in the formulations of both Theorem~\ref{existrhou} and Theorem~\ref{solvevartheta} that the terms $P-P_0$ and $\tilde P$ do not contain $\rho^\gamma$, as we would not be able to make sense of $\div u\,\rho^\gamma$. However since $n\leq \gamma/2$, we have no difficulty in handling $\div u\,\rho^n$. In particular, we may easily remove the $\rho^n$ terms from \eqref{almostenergy} through the use of renormalized solutions. Since $\rho\in L^2_{t,x}$ and $u\in H^1_{t,x}$, the classical theory of renormalized solutions, from \cite{DL} for example, shows that for any smooth, bounded function $f(\rho)$, we have in the sense of distributions that
\[
\partial_t f(\rho)+\div(u\,f(\rho))=(f(\rho)-f'(\rho)\,\rho)\,\div u.
\]
By integrating over $x$, we have that
\[
\int_{\TT^d}  f(\rho)\,dx=\int_{\TT^d}  f(\rho^0)\,dx+\int_0^t\int_{\TT^d}  (f(\rho)-f'(\rho)\,\rho)\,\div u(s,x)\,dx\,ds.
\]
Since $\rho\in L^\infty_t L^\gamma_x$ and $\div u\in L^2_{t,x}$, we may now apply this to a sequence $f_\eps$ with $f_\eps(x)\to x^n$ as $\eps\to 0$ and obtain
\[
\int_{\TT^d}  \rho^n\,dx=\int_{\TT^d}  (\rho^0)^n\,dx-(n-1)\,\int_0^t\int_{\TT^d}  \rho^n\,\div u(s,x)\,dx\,ds,
\]
which leads to the desired energy inequality
\[
\int_{\TT^d}  \mathcal{E}(\rho,\vartheta,m)\,dx\leq 
 \int_{\TT^d}  \mathcal{E}(\rho_0,\vartheta_0, m_0)\,dx.
\]
\bigskip

The last remaining point is to extend this solution on $[0,\ T]$ to a solution that is global in time. This is naturally achieved by repeating the fixed point argument starting from $T$. To do so, we highlight the conditions on the initial data that Theorem~\ref{existrhou} and Theorem~\ref{solvevartheta} require: one needs $\vartheta^0\in L^{\gamma_\vartheta}_x$, $\rho^0\in L^\gamma_x$ together with $\rho^0\,|u^0|^2\in L^1_x$. Equivalently, we can require $\mathcal{E}(\rho^0,u^0,\vartheta^0)<\infty$. Indeed from~\eqref{lowerboundg}, we have that
\[
\int_{\TT^d}  \rho\,\frac{|u|^2}{2}\,dx+\int_{\TT^d}  \rho^{\gamma}\,dx+\int _{\TT^d} \vartheta^{\gamma_\vartheta}\,dx\leq C\,\mathcal{E}(\rho_0,\vartheta_0,m_0).
\]
As seen earlier in the proof, the time of existence $T$ is a function of the various norms of the initial data or again equivalently $\mathcal{E}(\rho_0,\vartheta_0,m_0)$. From the propagation of energy, we have that $\mathcal{E}(\rho,\vartheta,m)|_{t=T}$ is dominated by $\mathcal{E}(\rho_0,\vartheta^0, m_0)$ and therefore the existence time $T$ can be chosen uniformly whether starting at $t=0$, $t=T$ or $t=2T$. This ensures global existence.

\bigskip

\noindent {\bf Acknowledgments.}  The first author is partially supported by the SingFlows project, grant ANR-18-CE40-0027. The first author want to thank members of the Department of Mathematics and Huck Institutes in Pennsylvania State University for his  Shapiro visit in fall 2022. The second author is  partially supported by NSF DMS Grants 2205694, 2219397 and 2049020. The third author is partially supported by the National Natural Science Foundation of China (No. 12101396 and 12161141004).


\end{document}